\documentclass [twoside,reqno,   12pt] {amsart}

\usepackage{hyperref}

\usepackage{amsfonts}
\usepackage{amssymb}
\usepackage{a4}
\usepackage{color}
\usepackage{amscd,amssymb,amsfonts}

\usepackage{picture,xcolor}

\usepackage{etoolbox}

\usepackage{tikz}

\newtheorem{thm}{Theorem}[section]

\newtheorem{lem}[thm]{Lemma}
\newtheorem{prop}[thm]{Proposition}

\newtheorem{rem}[thm]{Remark}

\theoremstyle{definition}

\numberwithin{equation}{section}

\newcommand{\C}{\mathbb{C}}

\renewcommand{\div}{\operatorname{div}}

\newcommand{\N}{\mathbb{N}}

\newcommand{\R}{\mathbb{R}}

\newcommand{\supp}{\operatorname{supp}}

\def\hat{\widehat}
\def\tilde{\widetilde}
\def \bfo {\begin {eqnarray*} }
\def \efo {\end {eqnarray*} }
\def \ba {\begin {eqnarray*} }
\def \ea {\end {eqnarray*} }
\def \beq {\begin {eqnarray}}
\def \eeq {\end {eqnarray}}
\def \supp {\hbox{supp }}

\def \p {\partial}

\def\hat{\widehat}
\def\tilde{\widetilde}
\def \bfo {\begin {eqnarray*} }
\def \efo {\end {eqnarray*} }
\def \ba {\begin {eqnarray*} }
\def \ea {\end {eqnarray*} }
\def \beq {\begin {eqnarray}}
\def \eeq {\end {eqnarray}}
\def \supp {\hbox{supp }}

\def \p {\partial}

\begin{document}

 \title[Partial data problems for quasilinear conductivity equations]{Partial data inverse problems for quasilinear conductivity equations}

\author[Kian]{Yavar Kian}

\address
        {Y. Kian, Aix Marseille Univ\\ 
        Universit\'e de Toulon, CNRS\\
        CPT, Marseille, France}

\email{yavar.kian@univ-amu.fr}

\author[Krupchyk]{Katya Krupchyk}
\address
        {K. Krupchyk, Department of Mathematics\\
University of California, Irvine\\
CA 92697-3875, USA }

\email{katya.krupchyk@uci.edu}

\author[Uhlmann]{Gunther Uhlmann}

\address
       {G. Uhlmann, Department of Mathematics\\
       University of Washington\\
       Seattle, WA  98195-4350\\
       USA\\
        and Institute for Advanced Study of the Hong Kong University of Science and Technology}
\email{gunther@math.washington.edu}

\maketitle

\begin{abstract}
We show that the knowledge of the Dirichlet--to--Neumann maps given on an arbitrary open non-empty portion of the boundary of a smooth domain in $\R^n$, $n\ge 2$, for classes of semilinear and quasilinear conductivity equations, determines the nonlinear conductivities uniquely. The main ingredient in the proof is a certain $L^1$--density result involving sums of products of gradients of harmonic functions which vanish on a closed proper subset of the boundary.  

\end{abstract}

\section{Introduction and statement of results}

Let $\Omega\subset \R^n$, $n\ge 2$, be a connected bounded open set with $C^\infty$ boundary.  Let $\omega\in\mathbb{S}^{n-1}=\{\omega\in  \R^n: |\omega|=1\}$ be fixed and let us consider the Dirichlet problem for the following isotropic quasilinear conductivity equation, 
\begin{equation}
\label{eq_int_2}
\begin{cases} 
\div (\gamma(x,u,\omega\cdot \nabla u)\nabla u)=0& \text{in}\quad \Omega,\\
u=\lambda+f & \text{on}\quad \p \Omega. 
\end{cases}
\end{equation} 
Here we assume that the function $\gamma: \overline{\Omega}\times \C\times \C\to \C$ satisfies the following conditions, \begin{itemize}
\item[(i)] the map $\C\times \C\ni (\tau, z)\mapsto \gamma(\cdot, \tau,z)$ is holomorphic with values in the H\"older space   $C^{1,\alpha}(\overline{\Omega})$ with some $0<\alpha<1$, 
\item[(ii)] $\gamma(x,\tau,0)=1$, for all $x\in \Omega$ and all $\tau\in \C$.
\end{itemize}
It is established in Theorem \ref{thm_well-posedness} that under the assumptions (i) and (ii) for each $\lambda\in \C$,  there exist  $\delta_\lambda>0$ and $C_\lambda>0$ such that when $f\in B_{\delta_\lambda}(\p \Omega):=\{f\in C^{2,\alpha}(\p \Omega): \|f\|_{C^{2,\alpha}(\p \Omega)}<\delta_\lambda\}$, the problem  \eqref{eq_int_2} has a unique solution $u=u_{\lambda, f}\in C^{2,\alpha}(\overline{\Omega})$ satisfying $\|u-\lambda\|_{C^{2,\alpha}(\overline{\Omega})}< C_\lambda\delta_\lambda$.

Let $\Gamma\subset \p\Omega$ be an arbitrary non-empty open subset of the boundary $\p \Omega$.  Associated to the problem \eqref{eq_int_2}, we define the partial Dirichlet--to--Neumann map 
\[
\Lambda_\gamma^{\Gamma}(\lambda+f)=(\gamma(x,u,\omega\cdot\nabla u)\p_\nu u)|_{\Gamma},
\]
where $f\in B_{\delta_\lambda}(\p \Omega)$ with $\supp(f)\subset \Gamma$ and $\lambda\in \C$.  Here $\nu$ is the unit outer normal to the boundary.

We are interested in the following inverse boundary problem for the quasilinear conductivity equation \eqref{eq_int_2}: given  the knowledge of the partial Dirichlet--to--Neumann map $\Lambda_\gamma^{\Gamma}$, determine the quasilinear conductivity $\gamma$ in $\overline{\Omega}\times \C\times \C$.   Our first main result gives a solution to this problem. 
\begin{thm}
\label{thm_main} 
Let $\Omega\subset \R^n$, $n\ge 2$, be a connected bounded open set with $C^\infty$ boundary, and let $\Gamma\subset \p \Omega$ be an arbitrary open non-empty subset of the boundary $\p \Omega$. Let $\omega\in \mathbb{S}^{n-1}$ be fixed. Assume that $\gamma_1, \gamma_2: \overline{\Omega}\times \C\times \C\to \C$ satisfy the assumptions (i) and (ii). Let $\Sigma\subset \C$ be a set which has a limit point in $\C$. Then if for all $ \lambda\in \Sigma$, we have
\[
\Lambda_{\gamma_1}^{\Gamma}(\lambda+f)=\Lambda_{\gamma_2}^{\Gamma}(\lambda+f), \quad\forall f\in B_{\delta_\lambda}(\p \Omega), \ \emph{\supp}(f)\subset \Gamma, 
\]
then  $\gamma_1=\gamma_2$ in $\overline{\Omega}\times \C\times \C$. 
\end{thm}

Next let us consider the  Dirichlet problem for the following isotropic semilinear conductivity equation, 
\begin{equation}
\label{eq_int_0_1}
\begin{cases} 
\div (\gamma(x,u)\nabla u)=0& \text{in}\quad \Omega,\\
u=f & \text{on}\quad \p \Omega. 
\end{cases}
\end{equation}
Here we assume that the function $\gamma: \overline{\Omega}\times \C \to \C$ satisfies the following conditions, \begin{itemize}
\item[(a)] the map $\C\ni \tau\mapsto \gamma(\cdot, \tau)$ is holomorphic with values in   $C^{1,\alpha}(\overline{\Omega})$ with some $0<\alpha<1$, 
\item[(b)] $\gamma(x, 0)=1$, for all $x\in \Omega$.
\end{itemize}
 It is shown in Theorem \ref{thm_well-posedness} that under the  assumptions (a) and (b),  there exist $\delta>0$ and $C>0$ such that when $f\in B_\delta(\p \Omega):=\{f\in C^{2,\alpha}(\p \Omega): \|f\|_{C^{2,\alpha}(\p \Omega)}<\delta\}$, the problem  \eqref{eq_int_0_1} has a unique solution $u=u_{ f}\in C^{2,\alpha}(\overline{\Omega})$ satisfying $\|u\|_{C^{2,\alpha}(\overline{\Omega})}<C\delta$. 

Associated to the problem \eqref{eq_int_0_1}, we define the partial Dirichlet--to--Neumann map 
\[
\Lambda_\gamma^{\Gamma}(f)=(\gamma(x,u)\p_\nu u)|_{\Gamma},
\]
where $f\in B_\delta(\p \Omega)$ with $\supp(f)\subset \Gamma$.   Our second main result is as follows. 
\begin{thm}
\label{thm_main_semilinear} 
Let $\Omega\subset \R^n$, $n\ge 2$, be a connected bounded open set with $C^\infty$ boundary, and let $\Gamma\subset \p \Omega$ be an arbitrary open non-empty subset of the boundary $\p \Omega$. Let $\gamma_1, \gamma_2: \overline{\Omega}\times \C\to \C$ satisfy the assumptions (a) and (b).  If $\Lambda_{\gamma_1}^\Gamma=\Lambda_{\gamma_2}^\Gamma$ then $\gamma_1=\gamma_2$ in $\overline{\Omega}\times \C$. 
\end{thm}

\begin{rem}
An analog of Theorem \ref{thm_main_semilinear} in the full data case, i.e. when $\Gamma=\p \Omega$, was proved in  \cite{Sun_96} where instead of working with small Dirichlet data one considers small perturbations of constant Dirichlet data as in \eqref{eq_int_2}. Furthermore, it was assumed in \cite{Sun_96}  that the semilinear conductivity  is strictly positive while no analyticity was required. The proof of  \cite{Sun_96} relies  on a first order linearization of the Dirichlet--to--Neumann map at constant Dirichlet boundary values which leads to the inverse boundary problem for the linear conductivity equation and therefore, an application of results of \cite{SU_1987}  and \cite{Nachman_1996} for the linear conductivity problem  in dimensions $n\ge 3$ and in dimension $n=2$, respectively, gives the recovery of the semilinear condutivity. 
\end{rem}

\begin{rem}
To the best of our knowledge Theorem \ref{thm_main}  is new even in the full data case. Indeed, in the full data case, so far authors have only considered the recovery of conductivities of the form $\gamma(x,u)$, see e.g. \cite{Sun_96,Sun_Uhlm_97},  or  of the form $\gamma(u, \nabla u)$, see e.g.  \cite{Mun_Uhl_2018}, \cite{Shankar_2019}, or conductivities which depend $x$ and $\nabla u$ in some specific way, see e.g. \cite{CFeiz_2020_1}.
We obtain in  Theorem \ref{thm_main}, for what seems to be the first time, the recovery of some general class of quasilinear conductivities of the form $\gamma(x,u, \omega\cdot \nabla u)$,  depending on the space variable, the solution, as well as the derivative of the solution in a fixed direction.
\end{rem}

\begin{rem} To the best of our knowledge, the partial data results of Theorem \ref{thm_main} and Theorem \ref{thm_main_semilinear} are the first partial data results for nonlinear conductivity equations. 
\end{rem}

\begin{rem} The assumption that the conductivity is holomorphic as a function $\C\ni \tau\mapsto \gamma(\cdot, \tau,\cdot)$ in Theorem \ref{thm_main} is motivated by the proof of the solvability of the forward problem and the differentiability with respect to the boundary data. This assumption could perhaps be weakened and one could show that the full knowledge of the partial Dirichlet--to--Neumman map $\Lambda_{\gamma}^{\Gamma}$ determines the conductivity $\gamma$.  As the main focus of this paper is on establishing the partial data inverse results, we decided not to elaborate upon this issue further.
\end{rem}

\begin{rem}
It might be interesting to note that an analog of the partial data results of Theorem \ref{thm_main} and Theorem \ref{thm_main_semilinear}  is still not known in the case of the linear conductivity equation  in dimensions $n\ge 3$.   We refer to  
\cite{Iman_Uh_Yam_2010} for the corresponding partial data result  for the linear conductivity equation in dimension $n=2$.
\end{rem}

We remark that starting with \cite{Kurylev_Lassas_Uhlmann_2018}, it has been known that nonlinearity may be helpful when solving inverse problems for hyperbolic PDE. Analogous phenomena for nonlinear elliptic equations have been revealed and exploited in \cite{Feizmohammadi_Oksanen}, \cite{LLLS}, see also \cite{LLLS_partial}, \cite{Krup_Uhlmann_non_linear_1}, \cite{Krup_Uhlmann_2}, \cite{Krup_Uhlmann_non_magnetic}, \cite{Lai_Ting}.  A noteworthy aspect of all of these works is that the presence of a nonlinearity enables one to solve inverse problems for nonlinear PDE in situations where the corresponding inverse problems for linear equations are still open. The present paper  is also concerned with illustrating this general phenomenon.

Let us proceed to discuss the main ideas of the proofs of   Theorem \ref{thm_main} and Theorem \ref{thm_main_semilinear}.   Using the technique of higher order linearizations of the partial Dirichlet--to--Neumann map, introduced in  \cite{Feizmohammadi_Oksanen}, \cite{LLLS}, see also \cite{Sun_96}, \cite{Sun_Uhlm_97} for the use of  the second linearization,  we reduce the proof of Theorem \ref{thm_main}  to the following density result.  
\begin{thm}
\label{thm_main_2}
Let $\Omega\subset \R^n$, $n\ge 2$, be a connected bounded open set with $C^\infty$ boundary, let $ \Gamma\subset \p \Omega$ be an open non-empty subset of $\p \Omega$, let $\omega\in \mathbb{S}^{n-1}$ be fixed, and let $m=2,3, \dots,$ be fixed. Let $f\in L^\infty(\Omega)$ be such that 
\begin{equation}
\label{eq_int_3}
\int_\Omega f \bigg(\sum_{k=1}^m  \prod_{r=1,r\ne k}^m (\omega\cdot \nabla u_r)\nabla u_k\bigg)\cdot \nabla u_{m+1}dx=0,
\end{equation}
for all  functions $u_l\in C^\infty(\overline{\Omega})$ harmonic in $\Omega$ with $\emph{\supp}(u_l|_{\p \Omega})\subset \Gamma$, $l=1,\dots, m+1$. Then $f=0$ in $\Omega$.  
\end{thm}

Similarly, using higher order linearizations of the partial Dirichlet--to--Neumann map, we show that Theorem \ref{thm_main_semilinear}  will follow from the following density result.
\begin{thm}
\label{thm_main_2_semilinear}
Let $\Omega\subset \R^n$, $n\ge 2$, be a connected bounded open set with $C^\infty$ boundary, let $ \Gamma\subset \p \Omega$ be an open non-empty subset of $\p \Omega$, and let $m=2,3, \dots,$ be fixed. Let $f\in L^\infty(\Omega)$ be such that 
\begin{equation}
\label{eq_int_3_semilinear}
\int_\Omega f \bigg(\sum_{k=1}^m  \prod_{r=1,r\ne k}^m u_r\nabla u_k\bigg)\cdot \nabla u_{m+1}dx=0,
\end{equation}
for all  functions $u_l\in C^\infty(\overline{\Omega})$ harmonic in $\Omega$ with $\emph{\supp}(u_l|_{\p \Omega})\subset \Gamma$, $l=1,\dots, m+1$. Then $f=0$ in $\Omega$.  
\end{thm}

Theorem \ref{thm_main_2} and Theorem \ref{thm_main_2_semilinear}  can be viewed as extensions of the results of \cite{DKSU} and \cite{Krup_Uhlmann_2}.  Indeed,  it was proved in \cite{DKSU} that the linear span of the set of products of harmonic functions in $\Omega$ which vanish on a closed proper subset of the boundary is dense in $L^1(\Omega)$, and this density result was extended  in  \cite{Krup_Uhlmann_2} by showing that the linear span of the set of scalar products of gradients of harmonic functions in  $\Omega$ which vanish on a closed proper subset of the boundary is also dense in $L^1(\Omega)$.

To prove  Theorem \ref{thm_main_2}, we shall follow the general strategy of the work \cite{DKSU}, see also  \cite{Krup_Uhlmann_2}. We first establish a corresponding local result in a neighborhood of a boundary point in $\Gamma$ assuming, as we may, that $\Gamma$ is a small open neighborhood of this point, see Proposition \ref{prop_local} below. We then show how to pass from this local result to the global one of  Theorem \ref{thm_main_2}. The essential difference here compared with the works \cite{DKSU},  \cite{Krup_Uhlmann_2} is that working with products of $m+1$ gradients in the orthogonally identity \eqref{eq_int_3}, we need to prove a certain  Runge type approximation theorem in the $W^{1,m+1}$--topology for any $m=2,3,\dots$ fixed, as opposed to $L^2$ and $H^1$ approximation results obtained in \cite{DKSU} and \cite{Krup_Uhlmann_2}, respectively.

We shall only prove  Theorem \ref{thm_main_2} as the proof of Theorem \ref{thm_main_2_semilinear} is obtained by inspection of that proof as  the only difference between the orthogonally relations  \eqref{eq_int_3}  and \eqref{eq_int_3_semilinear}  is that \eqref{eq_int_3} contains $\omega\cdot\nabla u_r$ with harmonic functions $u_r$ while 
\eqref{eq_int_3_semilinear} contains  $u_r$  instead, and no new difficulties occur.

Let us finally remark that inverse boundary problems for nonlinear elliptic PDE have been studied extensively in the literature. We refer to  \cite{CFeiz_2020_2}, \cite{CFeiz_2020_1}, \cite{CNV_2019}, \cite{Feizmohammadi_Oksanen},  \cite{Hervas_Sun}, \cite{Isakov_2001}, \cite{IsaNach_1995}, \cite{IsaSyl_94}, \cite{Kang_Nak_02},  \cite{Krup_Uhlmann_non_magnetic}, \cite{IY_2013}, \cite{LLLS}, \cite{Mun_Uhl_2018}, \cite{Shankar_2019}, \cite{Sun_96},  \cite{Sun_2004},   \cite{Sun_2010},  \cite{Sun_Uhlm_97},  and the references given there.  In particular, inverse boundary problems with partial data were studied  for  a certain class of semilinear equations of the form $-\Delta u +V(x,u)=0$ in \cite{Krup_Uhlmann_non_linear_1},   \cite{LLLS_partial} relying on the density result of \cite{DKSU},  for  semilinear equations of the form $-\Delta u+q(x)(\nabla u)^2=0$ in \cite{Krup_Uhlmann_2}, and for nonlinear magnetic Schr\"odinger equations in \cite{Lai_Ting}.

The paper is organized as follows. In Section \ref{sec_density} we establish Theorem \ref{thm_main_2}.  Theorem \ref{thm_main} in proven in Section \ref{sec_Thm_main}.  The proof of Theorem \ref{thm_main_semilinear} occupies Section \ref{sec_Thm_main_semilinear}.  In  Appendix \ref{sec_Thm_main_full_data} we present an alternative simple proof of Theorem \ref{thm_main}  in the full data case.  In Appendix \ref{App_well_posedness} we show the well-posedness of the Dirichlet problem for our quasilinear conductivity equation, in the case of boundary data close to a constant one.

\section{Proof of Theorem \ref{thm_main_2}}
\label{sec_density}
We shall proceed by following the general strategy of \cite{DKSU}.  It suffices to assume that $\Gamma\subset \p \Omega$ is a proper open nonempty subset of $\p \Omega$, and even a small open neighborhood of some boundary point.  

\subsection{Local result}
\label{subsection_local}

Theorem \ref{thm_main_2} will be obtained as a corollary of the following local result.

\begin{prop}
\label{prop_local}
Let $\Omega\subset \R^n$, $n\ge 2$, be a bounded open set with $C^\infty$ boundary, and let  $m=2,3, \dots,$ be fixed. Let $x_0\in \p \Omega$, and let $\tilde \Gamma\subset \p \Omega$ be the complement of an open boundary neighborhood of $x_0$. Then there exists $\delta>0$ such that if we have \eqref{eq_int_3}  for any harmonic functions $u_l\in C^\infty(\overline{\Omega})$ satisfying $u_l|_{\tilde \Gamma}=0$, $l=1,\dots, m+1$, then $f=0$ in $B(x_0, \delta)\cap \Omega$.
\end{prop}

\begin{proof}
It suffices to choose $u_1=\dots=u_{m}$ in \eqref{eq_int_3}. Hence,     \eqref{eq_int_3} implies that 
\begin{equation}
\label{eq_6_1}
\int_{\Omega} f(\omega\cdot \nabla v_1)^{m-1}\nabla v_1\cdot \nabla v_2 dx=0,   
\end{equation}
for all harmonic functions $v_1, v_2\in C^\infty(\overline{\Omega})$ satisfying $v_l|_{\tilde \Gamma}=0$, $l=1,2$.  Our goal is to show that  \eqref{eq_6_1} gives that $f=0$ in $B(x_0, \delta)\cap \Omega$ with $\delta>0$. Arguing as in \cite{DKSU}, we are reduced to the following setting: $x_0=0$, the tangent plane to $\Omega$ at $x_0$ is given by $x_1=0$,
\[
\Omega\subset \{ x\in \R^n: |x+e_1|<1\}, \quad \tilde \Gamma =\{x\in \p \Omega: x_1\le -2 c\}, \quad e_1=(1,0,\dots, 0),
\]
for some $c>0$. 

Let $p(\zeta)=\zeta^2$, $\zeta\in \C^n$, be the principal symbol of $-\Delta$, holomorphically extended to $\C^n$. Let $\zeta\in p^{-1}(0)$ and let $\chi \in C_0^\infty(\R^n)$ be such that $\chi=1$ on $\tilde \Gamma$. We shall work with harmonic functions of the form
\begin{equation}
\label{eq_6_2}
v(x,\zeta)=e^{-\frac{i}{h} x\cdot\zeta} +r(x,\zeta),
\end{equation}
where $r$ is the solution to the Dirichlet problem,
\[
\begin{cases} -\Delta r =0 \quad \text{in}\quad \Omega,\\
r|_{\p \Omega}=-(e^{-\frac{i}{h} x\cdot\zeta} \chi)|_{\p \Omega}.
\end{cases}
\]
By the boundary elliptic regularity, we have $v\in C^\infty(\overline{\Omega})$,  and furthermore $v|_{\tilde \Gamma}=0$. Since in view of \eqref{eq_6_1} we shall work with products of $m+1$ gradients of harmonic functions, we need to have  good estimates for the remainder $r$ in $C^1(\overline{\Omega})$. To that end, in view of Sobolev's embedding, we would like to bound $\|r\|_{H^{k}(\Omega)}$ with $k\in \N$,   $k>n/2+1$. Boundary elliptic regularity gives that for $k\ge 2$, 
\begin{equation}
\label{eq_6_3}
\|r\|_{H^{k}(\Omega)}\le C\|e^{-\frac{i}{h} x\cdot\zeta} \chi\|_{H^{k-1/2}(\p \Omega)}, 
\end{equation}
see \cite[Section 24.2]{Eskin_book}. Now by interpolation, we get 
\begin{equation}
\label{eq_6_4}
\|e^{-\frac{i}{h} x\cdot\zeta} \chi\|_{H^{k-1/2}(\p \Omega)}\le \|e^{-\frac{i}{h} x\cdot\zeta} \chi\|_{H^{k}(\p \Omega)}^{1/2}\|e^{-\frac{i}{h} x\cdot\zeta} \chi\|_{H^{k-1}(\p \Omega)}^{1/2},
\end{equation}
see \cite[Theorem 7.22, p. 189]{Grubb_book}. We have 
\[
\|e^{-\frac{i}{h} x\cdot\zeta} \chi \|_{L^2(\p \Omega)}\le Ce^{\frac{1}{h}\sup_{x\in K}x\cdot \text{Im}\, \zeta },
\]
where $K=\supp\chi\cap \p \Omega$, and therefore, 
\begin{equation}
\label{eq_6_5}
\|e^{-\frac{i}{h} x\cdot\zeta} \chi\|_{H^{k}(\p \Omega)}\le C\bigg(1+\frac{|\zeta|}{h}+\cdots+\frac{|\zeta|^k}{h^k}\bigg)e^{\frac{1}{h}\sup_{x\in K}x\cdot \text{Im}\,\zeta }.
\end{equation}
It follows from \eqref{eq_6_4} and \eqref{eq_6_5} that 
\begin{equation}
\label{eq_6_6}
\|e^{-\frac{i}{h} x\cdot\zeta} \chi\|_{H^{k-1/2}(\p \Omega)}\le C \bigg(1+\frac{|\zeta|^k}{h^k}\bigg)e^{\frac{1}{h}\sup_{x\in K}x\cdot \text{Im}\,\zeta }.
\end{equation}
Using \eqref{eq_6_3} and \eqref{eq_6_6}, we see that 
\[
\|r\|_{H^{k}(\Omega)}\le C \bigg(1+\frac{|\zeta|^k}{h^k}\bigg)e^{\frac{1}{h}\sup_{x\in K}x\cdot \text{Im}\, \zeta }.
\]
Taking $k>n/2+1$ and using the Sobolev embedding $H^k(\Omega)\subset C^1(\overline{\Omega})$, we get 
\begin{equation}
\label{eq_6_7}
\|r\|_{C^1(\overline{\Omega})}\le C \bigg(1+\frac{|\zeta|^k}{h^k}\bigg)e^{\frac{1}{h}\sup_{x\in K}x\cdot \text{Im}\,\zeta }.
\end{equation}

Choosing $\chi \in C_0^\infty(\R^n)$ so that $\supp(\chi)\subset \{x\in \R^n: x_1\le -c\}$ and $\chi=1$ on $\{x\in \p \Omega: x_1\le -2c\}$, we obtain from \eqref{eq_6_7} that
\begin{equation}
\label{eq_6_8}
\|r\|_{C^1(\overline{\Omega})}\le C \bigg(1+\frac{|\zeta|^k}{h^k}\bigg)  e^{-\frac{c}{h}\text{Im}\, \zeta_1} e^{\frac{1}{h}|\text{Im}\, \zeta'|},
\end{equation}
when $\text{Im}\,\zeta_1\ge 0$.

Now the identity \eqref{eq_6_1} implies that 
\begin{equation}
\label{eq_6_9}
\int_\Omega f(x) (\omega\cdot hDv(x,\zeta))^{m-1} hDv(x,\zeta)\cdot hD v(x,m\eta)dx=0,
\end{equation}
for all $\zeta,\eta\in p^{-1}(0)$. Here $v(x,\zeta)$ and $v(x,m\eta)$ are harmonic functions of the form \eqref{eq_6_2} and $D=i^{-1}\nabla$. 
Using that 
\begin{align*}
&(\omega\cdot hDv(x,\zeta))^{m-1} =(-\omega\cdot \zeta e^{-\frac{i}{h}x\cdot \zeta}+\omega\cdot hDr(x,\zeta))^{m-1}\\
&=
(-\omega\cdot \zeta)^{m-1}e^{-\frac{(m-1)i}{h}x\cdot\zeta}+\sum_{l=1}^{m-1}\begin{pmatrix} m-1\\l \end{pmatrix} (\omega\cdot hDr(x,\zeta))^l(-\omega\cdot \zeta e^{-\frac{i}{h}x\cdot\zeta})^{m-1-l},
\end{align*}
we obtain from \eqref{eq_6_9} that 
\begin{equation}
\label{eq_6_10}
\int_\Omega f(x) (-\omega\cdot \zeta)^{m-1} m (\zeta\cdot \eta)e^{-\frac{mi}{h}x\cdot(\zeta+\eta)}dx=I_1+I_2,
\end{equation}
where 
\begin{align*}
I_1=-\int_\Omega f(x) &(-\omega\cdot \zeta)^{m-1} e^{-\frac{(m-1)i}{h}x\cdot\zeta}\big(-\zeta e^{-\frac{i}{h}x\cdot \zeta}\cdot hDr(x,m\eta) \\
&-m\eta e^{-\frac{mi}{h}x\cdot \eta}\cdot hDr(x,\zeta)
+hDr(x,\zeta)\cdot hDr(x,m\eta)\big)dx,
\end{align*}
\begin{align*}
I_2=&-\int_\Omega f(x) \sum_{l=1}^{m-1}\begin{pmatrix} m-1\\l \end{pmatrix} (\omega\cdot hDr(x,\zeta))^l(-\omega\cdot \zeta e^{-\frac{i}{h}x\cdot\zeta})^{m-1-l}\\
&\big(m\zeta\cdot \eta e^{-\frac{i}{h}x\cdot(\zeta+m\eta)}
 -\zeta e^{-\frac{i}{h}x\cdot \zeta}\cdot hDr(x,m\eta) -m\eta e^{-\frac{mi}{h}x\cdot \eta}\cdot hDr(x,\zeta)\\
&+hDr(x,\zeta)\cdot hDr(x,m\eta)\big)dx.
\end{align*}
We shall next proceed to bound the absolute values of $I_1$ and $I_2$. To that end, first note that 
when $\text{Im}\,\zeta_1\ge 0$, using the fact that $\Omega\subset \{x\in \R^n: |x +e_1|<1\}$, we have
\begin{equation}
\label{eq_6_11}
 \|e^{-\alpha\frac{ix\cdot \zeta}{h}}\|_{L^\infty(\Omega)}\le e^{\alpha \frac{|\text{Im}\, \zeta' |}{h}}, \quad \alpha> 0.
\end{equation}
Using \eqref{eq_6_8} and  \eqref{eq_6_11}, we obtain that for all $\zeta,\eta\in p^{-1}(0)$, $\text{Im}\,\zeta_1\ge 0$, $\text{Im}\, \eta_1\ge 0$, 
\begin{equation}
\label{eq_6_12}
\begin{aligned}
|I_1|&\le C\|f\|_{L^\infty} e^{\frac{m(|\text{Im}\, \zeta' |+ |\text{Im}\, \eta'|)}{h}} e^{-\frac{c}{h}\min(\text{Im}\, \zeta_1, \text{Im}\, \eta_1)} |\zeta|^{m-1}\\
&\bigg( |\zeta| h\bigg(1+\frac{|m\eta|^k}{h^k}\bigg)+m|\eta|h\bigg(1+\frac{|\zeta|^k}{h^k}\bigg)+h^2\bigg(1+\frac{|m\eta|^k}{h^k}\bigg)
\bigg(1+\frac{|\zeta|^k}{h^k}\bigg) \bigg),
\end{aligned}
\end{equation}
and 
\begin{equation}
\label{eq_6_13}
\begin{aligned}
|I_2|\le C\|f\|_{L^\infty}& e^{\frac{m(|\text{Im}\, \zeta' |+ |\text{Im}\, \eta'|)}{h}} e^{-\frac{c}{h}\min(\text{Im}\, \zeta_1, \text{Im}\, \eta_1)} h \bigg(1+\frac{|\zeta|^k}{h^k}\bigg)^{m-1} (1+|\zeta|^{m-2})\\
&\bigg(m|\zeta| |\eta|+ |\zeta| h\bigg(1+\frac{|m\eta|^k}{h^k}\bigg)+m|\eta|h\bigg(1+\frac{|\zeta|^k}{h^k}\bigg)\\
&+h^2\bigg(1+\frac{|m\eta|^k}{h^k}\bigg)
\bigg(1+\frac{|\zeta|^k}{h^k}\bigg) \bigg).
\end{aligned}
\end{equation}

As noticed in \cite{DKSU}, the differential of the map
\[
s:p^{-1}(0)\times p^{-1}(0)\to \C^n, \quad (\zeta,\eta)\mapsto \zeta+\eta.
\]
at a point $(\zeta_0,\eta_0)$ is surjective, provided that $\zeta_0$ and $\eta_0$ are linearly independent. The latter holds if $\zeta_0=\gamma$ and $\eta_0=-\overline{\gamma}$ with $\gamma\in \C^n$ given as follows. Recall that $\omega=(\omega_1,\dots, \omega_n)\in \mathbb{S}^{n-1}$ is fixed. Then there exists $\omega_k\ne 0$, and if $2\le k\le n $ we set $\gamma=(i, 0, \dots, 0, 1, 0, \dots, 0)$ where $1$ is on the $k$th position. 
 If $\omega_1\ne 0$ then we set $\gamma=(i,1, 0, \dots, 0)\in \C^n$.

Note that $\gamma\cdot \omega\ne 0$ and $\zeta_0+\eta_0=2i e_1$.  An application of  the inverse function theorem gives  that there exists $\varepsilon>0$ small such that any $z\in \C^n$, $|z-2ie_1|<2\varepsilon$, may be decomposed as $z=\zeta+\eta$ where $\zeta, \eta\in p^{-1}(0)$, $|\zeta-\gamma|<C_1\varepsilon$ and $|\eta+\overline{\gamma}|<C_1\varepsilon$ with some $C_1>0$. We obtain that any $z\in \C^n$ such that  $|z-2i ae_1|<2\varepsilon a$ for some $a>0$, may be decomposed as
\begin{equation}
\label{eq_6_14}
z=\zeta+\eta, \quad \zeta, \eta\in p^{-1}(0),\quad  |\zeta-a \gamma|<C_1a \varepsilon, \quad |\eta+a\overline{\gamma}|<C_1a\varepsilon.
\end{equation}
It follows from \eqref{eq_6_14} that 
\begin{equation}
\label{eq_6_15}
|\text{Im}\,\zeta'|<C_1a\varepsilon, \quad |\text{Im}\,\eta'|<C_1a\varepsilon, \quad |\zeta|\le Ca,\quad |\eta|\le Ca.
\end{equation}
We also conclude from \eqref{eq_6_14}  that for  $\varepsilon>0$ is small enough, 
\begin{equation}
\label{eq_6_16}
\text{Im}\, \zeta_1>a/2, \quad \text{Im}\, \eta_1>a/2, \quad |\zeta\cdot\eta|\ge a^2, \quad |\omega\cdot \zeta|>\frac{a}{2}\sqrt{\omega_1^2+\omega_k^2}. 
\end{equation}
Hence, assuming that $a>1$, we obtain from \eqref{eq_6_10} with the help of \eqref{eq_6_12},  \eqref{eq_6_13},  \eqref{eq_6_14},  \eqref{eq_6_15}, \eqref{eq_6_16} that 
\begin{equation}
\label{eq_6_17}
\begin{aligned}
\bigg|\int_\Omega f(x) e^{-\frac{mi}{h}x\cdot z}dx\bigg|&\le C\|f\|_{L^\infty} e^{-\frac{ca}{2h}}e^{\frac{2mC_1 a\varepsilon}{h}}\bigg(\frac{a}{h}\bigg)^{N}\\
&\le C\|f\|_{L^\infty} e^{-\frac{ca}{4h}}e^{\frac{2mC_1 a\varepsilon}{h}},
\end{aligned}
\end{equation}
for all $z\in \C^n$ such that $|z-2i ae_1|<2\varepsilon a$  and  $\varepsilon>0$ sufficiently small. Here $N$ is a fixed integer which depends on $k$ and $m$.  The estimate \eqref{eq_6_17} is analogous to the bound (3.8) in \cite{DKSU}, and hence, the proof of Proposition \ref{prop_local} may be completed by proceeding as in \cite{DKSU}.
\end{proof}

Next in order to pass from this local result to the global one of Theorem \ref{thm_main_2}, we need  a Runge type approximation theorem in the $W^{1,m+1}$--topology, $m=2,3,\dots$, which will extend \cite[Lemma 2.2]{DKSU} and \cite[Lemma 2.2]{Krup_Uhlmann_2}, where approximations in the $L^2$ and $H^1$ topologies were established, respectively.  To prove such an approximation theorem, we need to recall some facts about $L^p$ based Sobolev spaces which we shall now proceed to do.

\subsection{Some facts about $L^p$ based Sobolev spaces}

Here we recall some definitions and facts regarding $L^p$ based Sobolev spaces following \cite{Brezis_book}, see also \cite{Triebel_book_1978}. 
Let $\Omega\subset \R^n$, $n\ge 2$, be an open set, and let $1<p<\infty$. The Sobolev space $W^{1,p}(\Omega)$ is defined by 
\[
W^{1,p}(\Omega)=\{u\in L^p(\Omega): \p_{x_j} u\in L^p(\Omega), \ 
 j=1,\dots,n\}.
\]
When equipped with the norm
\[
\|u\|_{W^{1,p}(\Omega)}=\|u\|_{L^p(\Omega)}+\sum_{j=1}^n\|\p_{x_j} u\|_{L^p(\Omega)},
\]
the space $W^{1,p}(\Omega)$ becomes a Banach space. We write $H^1(\Omega)=W^{1,2}(\Omega)$. 

The space $W^{1,p}_0(\Omega)$ is defined as the closure of $C^\infty_0(\Omega)$ in $W^{1,p}(\Omega)$. Since $C^\infty_0(\R^n)$ is dense in $W^{1,p}(\R^n)$, we have 
\[
W_0^{1,p}(\R^n)=W^{1,p}(\R^n). 
\]
We denote by $W^{-1,p'}(\Omega)$ the dual space of $W^{1,p}_0(\Omega)$ and write 
\[
(W^{1,p}_0(\Omega))^*=W^{-1,p'}(\Omega),
\]
where  $\frac{1}{p}+\frac{1}{p'}=1$. We have that if $u\in W^{-1,p'}(\Omega)$ then there exist $f_0,f_1,\dots, f_n\in L^{p'}(\Omega)$ such that 
\[
u= f_0+ \sum_{j=1}^n \p_{x_j} f_j, 
\]
and  
\[
\|u\|_{W^{-1,p'}(\Omega)}=\max_{0\le j\le n} \|f_j\|_{L^{p'}(\Omega)},
\]
see  \cite[Proposition 9.20]{Brezis_book}.

From now on let $\Omega\subset \R^n$, $n\ge 2$, be a bounded open set with $C^\infty$ boundary.  We have for the dual space of $W^{1,p}(\Omega)$, 
\[
(W^{1,p}(\Omega))^*=\tilde W^{-1,p'}(\Omega),
\]
where 
\[
 \tilde W^{-1,p'}(\Omega)=\{u\in W^{-1,p'}(\R^n): \supp(u)\subset \overline{\Omega}\},
\]
see  \cite[page 163]{Browder_61},  \cite[Section 4.3.2]{Triebel_book_1978}.  The duality pairing is defined as follows: 
if $v\in \tilde W^{-1,p'}(\Omega)$ and $u\in W^{1,p}(\Omega)$, we set 
\begin{equation}
\label{eq_1_1_0}
(v,u)_{\tilde W^{-1,p'}(\Omega), W^{1,p}(\Omega)}:=(v, \text{Ext}(u))_{W^{-1,p'}(\R^n), W^{1,p}(\R^n)},
\end{equation}
where $\text{Ext}(u)\in W^{1,p}(\R^n)$ is an arbitrary extension of $u$, see  \cite[Theorem 9.7]{Brezis_book} for the existence of such an extension, and  $(\cdot, \cdot)_{W^{-1,p'}(\R^n), W^{1,p}(\R^n)}$ is the extension of $L^2$ scalar product $(\varphi, \psi)_{L^2(\R^n)}=\int_{\R^n} \varphi(x)\overline{\psi(x)}dx$. One can show that the definition \eqref{eq_1_1_0} is independent of the choice of an extension.

We shall also need the following fact, see \cite[Section 4.3.2, p. 318]{Triebel_book_1978}.
\begin{prop}
\label{prop_density_in_tilde_space}
$C^\infty_0(\Omega)$ is dense in $\tilde W^{-1,p'}(\Omega)$ with respect to $W^{-1,p'}(\R^n)$ topology. 
\end{prop}

We define the fractional Sobolev space $W^{s,p}(\Omega)$, $0<s<1$, $1<p<\infty$, as follows, 
\[
W^{s,p}(\Omega)=\{u\in L^p(\Omega): \frac{|u(x)-u(y)|}{|x-y|^{s+n/p}}\in L^p(\Omega\times \Omega)\},
\]
equipped with the natural norm. By local charts, we define $W^{s, p}(M)$, $0<s<1$, $1<p<\infty$, when $M$ is a $C^\infty$ compact manifold without boundary. If $u\in W^{1,p}(\Omega)$ then $u|_{\p \Omega}\in W^{1-1/p,p}(\p \Omega)$ and 
\[
\|u|_{\p\Omega}\|_{W^{1-1/p,p}(\p \Omega)}\le C\|u\|_{W^{1,p}(\Omega)},  \text{ for all }u\in W^{1,p}(\Omega).
\]
Furthermore, the trace operator $W^{1,p}(\Omega)\ni u\mapsto u|_{\p \Omega}\in W^{1-1/p,p}(\p \Omega)$ is surjective, 
see \cite[p. 315]{Brezis_book}. We also have 
\begin{equation}
\label{eq_1_1_1}
W^{1,p}_0(\Omega)=\{u\in W^{1,p}(\Omega): u|_{\p \Omega}=0\},
\end{equation}
see \cite[p. 315]{Brezis_book}. We have another characterization of $W^{1,p}_0(\Omega)$, see  \cite[Proposition 9.18]{Brezis_book}. 
\begin{prop}
\label{prop_char_W_1_p_0}
Let $u\in L^p(\Omega)$, $1<p<\infty$. Then $u\in W^{1,p}_0(\Omega)$ if and only if the function 
\[
\tilde u=\begin{cases} u(x), & \text{if}\quad x\in\Omega,\\
0,&  \text{if}\quad x\in \R^n\setminus \Omega,
\end{cases}
\]
belongs to $W^{1,p}(\R^n)$. 
\end{prop}

We have the following result concerning the solvability of the Dirichlet problem for the Laplacian, see \cite[Theorem 7.10.2, p. 494]{Medkova_book}. 
\begin{thm}
\label{thm_solvability_Dirichlet}
Let $v\in W^{-1,p}(\Omega)$ and $g\in W^{1-1/p, p}(\p \Omega)$ with $1<p<\infty$. Then the Dirichlet problem
\[
\begin{cases}
-\Delta u=v & \text{in}\quad \Omega,\\
u|_{\p \Omega}=g, 
\end{cases}
\]
has a unique solution $u\in W^{1,p}(\Omega)$. Moreover, 
\[
\|u\|_{W^{1,p}(\Omega)}\le C(\|v\|_{W^{-1,p}(\Omega)}+\|g\|_{W^{1-1/p, p}(\p \Omega)}).
\]

\end{thm}

We shall also need the following result about the structure of distributions in $W^{-1,p}(\R^n)$ supported by a smooth hypersurface in $\R^n$. We refer to \cite[Theorem 5.1.13]{Agranovich_book},  \cite[Lemma 3.39]{McLean_book} for this result in the case of distributions in $H^{-1}(\R^n)$. Since we did not find a reference for the case of distributions in $W^{-1,p}(\R^n)$ with $1<p<\infty$, we shall present the proof of this result here. 

\begin{prop}
\label{prop_distributions_support_hyperplane}
Let $F$ be  a smooth compact hypersurface in $\R^n$. Let $u\in W^{-1,p}(\R^n)$, with some $1<p<\infty$, be such that $\supp(u)\subset F$. Then 
\[
u=v \otimes\delta_F, \quad v\in (W^{1-1/p',p'}(F))^*= B_{p,p}^{-(1-1/p')}(F).
\]
Here $\frac{1}{p}+\frac{1}{p'}=1$ and $B_{p,p}^{-(1-1/p')}(F)$ is the Besov space on the manifold $F$, see \cite[Section 2.3.1, p. 169]{Triebel_book_1978}, \cite{Triebel_1986}  for the definition. 

\end{prop}

\begin{proof}
Introducing a partition of unity and making a smooth change of variables, we see that it suffices to establish the following local result: let $u\in W^{-1, p}(\R^n)$, $1<p<\infty$, such that $\supp(u)\subset \{x_n=0\}$, then $u=v\otimes \delta_{x_n=0}$ with $v\in (W^{1-1/p',p'}(\R^{n-1}))^*= B_{p,p}^{-(1-1/p')}(\R^{n-1})$. In order to prove this result we follow  \cite[Lemma 3.39]{McLean_book}. 

First we claim that if $\varphi\in C^\infty_0(\R^n)$ is such that $\varphi|_{x_n=0}=0$ then $u(\varphi)=0$.  To that end, we let 
\[
\varphi_{\pm}(x)=\begin{cases} \varphi(x), & \text{if}\quad x\in \R^n_{\pm}=\{x\in \R^n: \pm x_n>0\},\\
0, & \text{otherwise}.
\end{cases}
\]
Then $\varphi_{\pm}\in W^{1,p'}(\R^n)$ and therefore, by Proposition \ref{prop_char_W_1_p_0} $\varphi_{\pm}\in W^{1,p'}_0(\R^n_{\pm})$.  Thus, there exist sequences $\varphi_{j,\pm}\in C^\infty_0(\R^n_{\pm})$ such that $\varphi_{j,\pm}\to \varphi_{\pm}$ in $W^{1,p'}(\R^n_{\pm})$ as $j\to\infty$.  Letting 
\[
\chi_j(x)=\begin{cases} \varphi_{j,+}(x), & \text{if}\quad x\in \R^n_+,\\
\varphi_{j,-}(x), & \text{if}\quad x\in \R^n_-,
\end{cases}
\]
we see that $\chi_j\in C^\infty_0(\R^n)$, $\chi_j=0$ near $\{x_n=0\}$, and $\chi_j\to\varphi$ in $W^{1,p'}(\R^n)$.  Hence, we have
$0=u(\chi_j)\to u(\varphi)$, and therefore, $u(\varphi)=0$, establishing the claim. 

To proceed we need the following result, see \cite{Mironescu_2015}, \cite[Theorem 1.5.1.1, p. 37]{Grisvard_book}.  The trace operator $u\mapsto u|_{x_n=0}$, which is defined on $C_0^\infty(\R^n)$, has a unique continuous extension as an operator,
\[
\gamma: W^{1,p'}(\R^n)\to W^{1-1/p',p'}(\R^{n-1}), \quad 1<p'<\infty.
\]
This operator has a right continuous inverse, the extension operator, 
\[
E: W^{1-1/p',p'}(\R^{n-1})\to W^{1,p'}(\R^n)
\]
so that $\gamma(E\psi)=\psi$ for all $\psi\in W^{1-1/p',p'}(\R^{n-1})$.  

Now we define 
\begin{equation}
\label{eq_def_v}
v(\varphi)=u(E\varphi), \quad \varphi\in C^\infty_0(\R^{n-1}). 
\end{equation}
We have
\[
|v(\varphi)|\le \|u\|_{W^{-1,p}(\R^n)}\|E\varphi\|_{W^{1,p'}(\R^n)}\le C\|u\|_{W^{-1,p}(\R^n)}\|\varphi\|_{W^{1-1/p',p'}(\R^{n-1})},
\]
and therefore, $v\in (W^{1-1/p',p'}(\R^{n-1}))^*$. Note that when $1<p'<\infty$, 
\[
W^{1-1/p',p'}(\R^{n-1})=B^{1-1/p'}_{p',p'}(\R^{n-1}), \quad (B^{1-1/p'}_{p',p'}(\R^{n-1}))^*=B^{-(1-1/p')}_{p,p}(\R^{n-1}),  
\]
see \cite[Section 2.5, p. 190, and Section 2.6.1, p. 198]{Triebel_book_1978}.

Finally, we claim that $u-v\otimes\delta_{x_n=0}=0$. Indeed, letting $\varphi\in C^\infty_0(\R^n)$ and using \eqref{eq_def_v} and our first claim, we get 
\[
(u-v\otimes\delta_{x_n=0})(\varphi)=u(\varphi)-v(\varphi|_{x_n=0})=u(\varphi-E(\varphi|_{x_n=0}))=0.
\]
This completes the proof of Proposition \ref{prop_distributions_support_hyperplane}. 
\end{proof}

\subsection{Runge type approximation}

\label{subsection_density_gradients} 

Let $\Omega_1\subset \Omega_2\subset \R^n$, $n\ge 2$,  be two bounded open sets with smooth boundaries such that
$\Omega_2\setminus\overline{\Omega_1}\ne\emptyset$. Suppose that $\p \Omega_1\cap \p \Omega_2=\overline{U}$ where $U\subset \p \Omega_1$ is open with $C^\infty$ boundary. Let $\mathcal{G}: C^\infty(\overline{\Omega_2})\to C^\infty(\overline{\Omega_2})$, $a\mapsto w$, be the solution operator to the Dirichlet problem,
\[
\begin{cases} -\Delta w=a & \text{in}\quad \Omega_2,\\
w|_{\p \Omega_2}=0.
\end{cases}
\]

The following result is an extension of \cite[Lemma 2.2]{DKSU} and \cite[Lemma 2.2]{Krup_Uhlmann_2}, where the similar density results were obtained in the $L^2$ and $H^1$ topologies, respectively.  
\begin{lem}
\label{lem_Runge}
 The space
\[
W:=\{\mathcal{G} a|_{\Omega_1}: a\in C^\infty(\overline{\Omega_2}),\ \emph{\supp}(a)\subset \Omega_2\setminus\overline{\Omega_1}\}
\]
is dense in the space
\[
S:=\{u\in C^\infty(\overline{\Omega_1}): -\Delta u=0 \text{ in }\Omega_1, \ u|_{\p \Omega_1\cap \p \Omega_2}=0\},
\]
with respect to  the $W^{1,p}(\Omega_1)$--topology, for any $1<p<\infty$.
\end{lem}

\begin{proof}
We shall follow the proof of \cite[Lemma 2.2]{Krup_Uhlmann_2} closely, adapting it to the $L^p$ based Sobolev spaces.  Let $v\in \tilde W^{-1,p'}(\Omega_1)$,  $\frac{1}{p}+\frac{1}{p'}=1$, be such that
\begin{equation}
\label{eq_1_1}
(v, \mathcal{G}a|_{\Omega_1})_{\tilde W^{-1,p'}(\Omega_1), W^{1,p}(\Omega_1)}= 0
\end{equation}
 for any $a\in C^\infty(\overline{\Omega_2})$, $\supp(a)\subset \Omega_2\setminus\overline{\Omega_1}$. In view of the Hahn--Banach theorem, we have to prove that 
\[ 
 (v, u)_{\tilde W^{-1,p'}(\Omega_1), W^{1,p}(\Omega_1)}=0,
\] 
 for any $u\in S$.

To that end, we first note that as $\mathcal{G}a\in C^\infty(\overline{\Omega_2})$ and $\mathcal{G}a|_{\p \Omega_2}=0$, it follows from \eqref{eq_1_1_1} that $\mathcal{G}a\in W^{1,p}_0(\Omega_2)$.  By Proposition \ref{prop_char_W_1_p_0}, we can view $\mathcal{G}a$ as an element of  $W^{1,p}(\R^n)$ via an extension by 0 to $\R^n\setminus \Omega_2$.  By the definition of $W^{1,p}_0(\Omega_2)$, there exists a sequence $\varphi_j\in C^\infty_0(\Omega_2)$ such that $\varphi_j\to \mathcal{G}a$ in $W^{1,p}(\R^n)$. We have in view of \eqref{eq_1_1}  that
\begin{equation}
\label{eq_1_2}
\begin{aligned}
0=(v, \mathcal{G}a)_{W^{-1,p'}(\R^n), W^{1,p}(\R^n)}= \lim_{j\to \infty} (v, \varphi_j)_{W^{-1,p'}(\R^n), W^{1,p}(\R^n)}\\=\lim_{j\to \infty} (v, \varphi_j)_{W^{-1,p'}(\Omega_2),W^{1,p}_0(\Omega_2)}=
(v, \mathcal{G}a)_{W^{-1,p'}(\Omega_2),W^{1,p}_0(\Omega_2)}.
\end{aligned}
\end{equation}

Next, Proposition \ref{prop_density_in_tilde_space} implies that there is a sequence $v_j\in C^\infty_0(\Omega_1)$ such that $v_j\to v$ in $W^{-1,p'}(\R^n)$. Consider the following  Dirichlet problems,
\begin{equation}
\label{eq_1_3}
\begin{cases}
-\Delta f=v|_{\Omega_2}\in W^{-1,p'}(\Omega_2) & \text{in}\quad \Omega_2,\\
f=0 & \text{on}\quad \p \Omega_2,
\end{cases}\quad \begin{cases}
-\Delta f_j=v_j & \text{in}\quad \Omega_2,\\
f_j=0 & \text{on}\quad \p \Omega_2.
\end{cases}
\end{equation}
By Theorem \ref{thm_solvability_Dirichlet} and \eqref{eq_1_1_1},  the problems \eqref{eq_1_3} have unique solutions $f\in W^{1,p'}_0(\Omega_2)$ and $f_j\in C^\infty(\overline{\Omega_2})\cap W^{1,p'}_0(\Omega_2)$, respectively. 

Using \eqref{eq_1_2}, \eqref{eq_1_3}, we get 
\begin{equation}
\label{eq_1_4}
\begin{aligned}
0&=(v, \mathcal{G}a)_{W^{-1,p'}(\Omega_2),W^{1,p}_0(\Omega_2)}=\lim_{j\to \infty} (v_j, \mathcal{G}a)_{W^{-1,p'}(\Omega_2),W^{1,p}_0(\Omega_2)}\\
 &=\lim_{j\to \infty} (-\Delta f_j, \mathcal{G}a)_{W^{-1,p'}(\Omega_2),W^{1,p}_0(\Omega_2)}
 = \lim_{j\to \infty}\int_{\Omega_2} (-\Delta f_j) \overline{\mathcal{G}a} dx\\
 &=\lim_{j\to \infty}\int_{\Omega_2}  f_j \overline{a} dx= \int_{\Omega_2}f\overline{a}dx.
\end{aligned}
\end{equation}
Here we have used Green's formula,  the fact that $f_j|_{\p \Omega_2}=\mathcal{G}a|_{\p \Omega_2}=0$, and that 
\[
\|f-f_j\|_{W^{1,p'}(\Omega_2)}\le C\|v-v_j\|_{W^{-1,p'}(\R^n)},
\]
which is a consequence of  Theorem \ref{thm_solvability_Dirichlet}.

It follows from \eqref{eq_1_4} that $f=0$ in $\Omega_2\setminus \overline{\Omega_1}$. This together with the fact that  $f\in W^{1,p'}_0(\Omega_2)$, in view of Proposition \ref{prop_char_W_1_p_0},  allows us to conclude that  $f\in W^{1,p'}_0(\Omega_1)$. Thus, there exists a sequence $\hat f_j\in C^\infty_0(\Omega_1)$ be such that $\hat f_j\to f$ in $W^{1,p'}(\R^n)$, and therefore,   $-\Delta\hat f_j\to -\Delta f $ in $W^{-1,p'}(\R^n)$.  

Let $u\in S$ and let $\text{Ext}(u)\in W^{1,p}(\R^n)$ be an extension of $u$. Using Green's formula, we get
\begin{equation}
\label{eq_1_5}
\begin{aligned}
 (-\Delta f, \text{Ext}(u))_{W^{-1,p'}(\R^n), W^{1,p}(\R^n)}
&=\lim_{j\to \infty}((-\Delta \hat f_j), \text{Ext}(u))_{W^{-1,p'}(\R^n), W^{1,p}(\R^n)}\\
&=\lim_{j\to \infty} \int_{\Omega_1} (-\Delta \hat f_j) \overline{u}dx=0.
\end{aligned}
\end{equation}

Let $g=-\Delta f-v\in W^{-1,p'}(\R^n)$.  We have that $\supp(g)\subset \p \Omega_1$, in view of the fact that  $\supp(v), \supp(f)\subset \overline{\Omega_1}$,  and \eqref{eq_1_3}. An application of Proposition \ref{prop_distributions_support_hyperplane} gives therefore
\[
g=h\otimes\delta_{\p \Omega_1}, \quad h\in B^{-(1-1/p)}_{p',p'}(\p \Omega_1).
\]
It also follows from \eqref{eq_1_3} that $\supp(g)\subset \p \Omega_1\cap \p \Omega_2=\overline{U}$, and hence, $\supp(h)\subset \overline{U}$. Here $U\subset \p \Omega_1$ is a bounded open set with $C^\infty$ boundary, and therefore, there exists a sequence $h_j\in C^\infty_0(U)$ such that $h_j\to h$ in $B^{-(1-1/p)}_{p',p'}(\p \Omega_1)$, see  \cite[Section 4.3.2, p. 318]{Triebel_book_1978}. Thus, we get 
\begin{equation}
\label{eq_1_6}
\begin{aligned}
 &(g, \text{Ext}(u))_{W^{-1,p'}(\R^n), W^{1,p}(\R^n)}= (h, u|_{\p \Omega_1})_{B^{-(1-1/p)}_{p',p'}(\p \Omega_1), W^{1-1/p,p}(\p \Omega_1)}\\
 &=\lim_{j\to \infty}  (h_j, u|_{\p \Omega_1})_{B^{-(1-1/p)}_{p',p'}(\p \Omega_1), B^{1-1/p}_{p,p}(\p \Omega_1)}=\lim_{j\to \infty}  \int_{\p \Omega_1} h_j \overline{u} dS=0,
\end{aligned}
\end{equation}
where  the last equality follows from the fact that $u|_{\p \Omega_1\cap \p \Omega_2}=0$.  Combining \eqref{eq_1_5} and \eqref{eq_1_6}, we see  that
\begin{align*}
(v, &u)_{\tilde W^{-1,p'}(\Omega_1), W^{1,p}(\Omega_1)} \\
&=(-\Delta f, \text{Ext}(u))_{W^{-1, p'}(\R^n), W^{1,p}(\R^n)} -(g, \text{Ext}(u))_{W^{-1,p'}(\R^n), W^{1,p}(\R^n)} =0.
\end{align*}
 \end{proof}

\subsection{From local to global results. Completion of proof of Theorem \ref{thm_main_2}} 
\label{subsection_global}
We follow \cite{DKSU}.  Let $\tilde \Gamma=\p \Omega\setminus \Gamma$.  Assuming  that $f$ satisfies \eqref{eq_int_3} and using Proposition \ref{prop_local}, we would like to show that $f$ vanishes inside $\Omega$. To that end,  let $x_0\in  \Gamma$   and let us fix a point $x_1\in \Omega$.  Let $\theta:[0,1]\to \overline{\Omega}$ be a $C^1$ curve joining $x_0$ to $x_1$ such that $\theta(0)=x_0$, $\theta'(0)$ is the interior normal to $\p \Omega$ at $x_0$ and $\theta(t)\in \Omega$, for all $t\in (0,1]$.  We set
\[
\Theta_\varepsilon(t)=\{x\in \overline{\Omega}: d(x, \theta([0,t]))\le \varepsilon\}
\]
and 
\[
I=\{t\in [0,1]: f\text{ vanishes a.e. on } \Theta_\varepsilon(t)\cap \Omega\}.
\]
By Proposition \ref{prop_local},  we have $0\in I$ if $\varepsilon>0$ is small enough. As explained in \cite{DKSU}, it suffices to prove that the set $I$ is open in $[0,1]$. 

To this end, let $t\in I$ and $\varepsilon>0$ be small enough so that $\p \Theta_\varepsilon(t)\cap \p \Omega\subset \Gamma$. Arguing as in  \cite{Krup_Uhlmann_2}, \cite{DKSU}, we smooth out  $\Omega\setminus \Theta_\varepsilon(t)$ into an open subset $\Omega_1$ of $\Omega$  with smooth boundary such that 
\[
\Omega_1\supset \Omega\setminus \Theta_\varepsilon(t), \quad \p \Omega\cap \p\Omega_1\supset \tilde \Gamma,
\]
and $\p \Omega_1\cap \p \Omega=\overline{U}$ where $U\subset \p \Omega_1$ is an open set with $C^\infty$ boundary. By smoothing out the set $\Omega\cup B(x_0, \varepsilon')$, with $0<\varepsilon'\ll \varepsilon$ sufficiently small, we enlarge  the set $\Omega$  into an open set $\Omega_2$ with smooth boundary so that
\[
\p \Omega_2\cap\p \Omega\supset \p \Omega_1\cap\p \Omega=\p \Omega_1\cap \p\Omega_2 \supset \tilde \Gamma.
\]

Let $G_{\Omega_2}$ be the Green kernel associated to the open set $\Omega_2$,
\[
-\Delta_y G_{\Omega_2}(x,y)=\delta(x-y), \quad G_{\Omega_2}(x,\cdot)|_{\p \Omega_2}=0,
\]
and let us consider 
\begin{align*}
&v(x^{(1)}, \dots, x^{(m+1)})=\\
&\int_{\Omega_1} f(y) \bigg(\sum_{k=1}^m  \prod_{r=1,r\ne k}^m (\omega\cdot \nabla_y G_{\Omega_2}(x^{(r)},y))\nabla_y G_{\Omega_2}(x^{(k)},y)\bigg)\cdot \nabla_y G_{\Omega_2}(x^{(m+1)},y)dy,
\end{align*}
where $x^{(1)}, \dots, x^{(m+1)}\in \Omega_2\setminus\overline{\Omega_1}$. The function $v$ is harmonic in all variables $x^{(1)}, \dots, x^{(m+1)}\in \Omega_2\setminus\overline{\Omega_1}$.  Since $f=0$ on $\Theta_\varepsilon(t)\cap \Omega$, we have
\begin{align*}
&v(x^{(1)}, \dots, x^{(m+1)})=\\
&\int_{\Omega} f(y) \bigg(\sum_{k=1}^m  \prod_{r=1,r\ne k}^m (\omega\cdot \nabla_y G_{\Omega_2}(x^{(r)},y))\nabla_y G_{\Omega_2}(x^{(k)},y)\bigg)\cdot \nabla_y G_{\Omega_2}(x^{(m+1)},y)dy,
\end{align*}
where $x^{(1)}, \dots, x^{(m+1)}\in \Omega_2\setminus\overline{\Omega_1}$.  Now when $x^{(l)}\in \Omega_2\setminus\overline{\Omega}$, the Green function $G_{\Omega_2} (x^{(l)},\cdot )\in C^\infty(\overline{\Omega})$ is  harmonic on $\Omega$, and $G_{\Omega_2} (x^{(l)},\cdot )|_{\tilde \Gamma}=0$.  By the orthogonality condition \eqref{eq_int_3}, we have $v(x^{(1)},\dots, x^{(m+1)})=0$ when $x^{(l)}\in \Omega_2\setminus\overline{\Omega}$, $l=1,\dots, m+1$.

As  $v(x^{(1)},\dots, x^{(m+1)})$  is harmonic in all variables $x^{(1)}, \dots, x^{(m+1)}\in \Omega_2\setminus\overline{\Omega_1}$, and $\Omega_2\setminus\overline{\Omega_1}$ is connected, by unique continuation, we get that  $v(x^{(1)},\dots, x^{(m+1)})=0$ when $x^{(1)}, \dots, x^{(m+1)}\in \Omega_2\setminus\overline{\Omega_1}$, i.e.
\begin{equation}
\label{eq_1_8}
\begin{aligned}
\int_{\Omega_1} f(y) \bigg(\sum_{k=1}^m  \prod_{r=1,r\ne k}^m (\omega\cdot \nabla_y G_{\Omega_2}(x^{(r)},y))\nabla_y G_{\Omega_2}(x^{(k)},y)\bigg)\cdot \nabla_y G_{\Omega_2}(x^{(m+1)},y)dy\\
=0, \quad x^{(1)},\dots, x^{(m+1)}\in \Omega_2\setminus\overline{\Omega_1}.
\end{aligned}
\end{equation}

Let $a_l\in C^\infty(\overline{\Omega_2})$, $\supp(a_l)\subset \Omega_2\setminus\overline{\Omega_1}$, $l=1,\dots,m$. Multiplying \eqref{eq_1_8} by $a_1(x^{(1)})\cdots a_m(x^{(m+1)})$,  and integrating, we get 
\begin{align*}
\int_{\Omega_1} f(y) \bigg(\sum_{k=1}^m  \prod_{r=1,r\ne k}^m \int_{\Omega_2}& (\omega\cdot \nabla_y G_{\Omega_2}(x^{(r)},y))a_r(x^{(r)})dx^{(r)}\\
&\int_{\Omega_2}
\nabla_y G_{\Omega_2}(x^{(k)},y) a_k(x^{(k)})dx^{(k)} \bigg)\\
&\cdot  \int_{\Omega_2}\nabla_y G_{\Omega_2}(x^{(m+1)},y)a_{m+1}(x^{(m+1)})dx^{(m+1)}dy
=0.
\end{align*}
Thus, we have 
\begin{equation}
\label{eq_1_9}
\begin{aligned}
\int_{\Omega_1} f(y) \bigg(\sum_{k=1}^m  \prod_{r=1,r\ne k}^m (\omega\cdot \nabla v^{(r)})\nabla v^{(k)}\bigg)\cdot \nabla v^{(m+1)}dy=0,
\end{aligned}
\end{equation}
for all $v^{(1)}, \dots, v^{(m)}\in W$, where $W$ is defined in Lemma \ref{lem_Runge}. 

The $(m+1)$--linear form,
\begin{align*}
&W^{1,m+1}(\Omega_1)\times \dots\times W^{1,m+1}(\Omega_1)\to \C, \\ 
&(v^{(1)}, \dots, v^{(m)})\mapsto \int_{\Omega_1} f(y) \bigg(\sum_{k=1}^m  \prod_{r=1,r\ne k}^m (\omega\cdot \nabla v^{(r)})\nabla v^{(k)}\bigg)\cdot \nabla v^{(m+1)}dy
\end{align*}
is continuous in view of  H\"older's inequality. 
An application of Lemma \ref{lem_Runge} with $p=m+1$ shows that  \eqref{eq_1_9} holds for all  $v^{(1)}, \dots, v^{(m)}\in C^\infty(\overline{\Omega_1})$ harmonic in $\Omega_1$ which vanish on $\p\Omega_1\cap \p\Omega_2$. 
 Proposition \ref{prop_local} implies that $f$ vanishes on a neighborhood of $\p \Omega_1\setminus(\p \Omega_1\cap \p \Omega_2)$, and therefore, $I$ is an open set. The proof of  Theorem \ref{thm_main_2} is complete.

\section{Proof of Theorem \ref{thm_main}}
\label{sec_Thm_main}

First it follows from (i) and (ii) that for each $\lambda\in \C$ fixed,   $\gamma$ can be expanded into a power series 
\begin{equation}
\label{eq_int_1}
\gamma(x,\lambda, z)=1+\sum_{k=1}^\infty \p^k_z\gamma(x,\lambda,0)\frac{z^k}{k!}, \quad \p^k_z\gamma(x,\lambda,0)\in C^{1,\alpha}(\overline{\Omega}), \quad \lambda, z\in \C, 
\end{equation}
converging in the $C^\alpha(\overline{\Omega})$ topology.  Furthermore, the map $\C\ni \lambda\mapsto \p^k_z\gamma(x,\lambda,0)$ is holomorphic with values in $C^\alpha(\overline{\Omega})$.

Let $\varepsilon=(\varepsilon_1, \dots, \varepsilon_m)\in \C^m$, $m\ge 2$, and consider the Dirichlet problem 
\eqref{eq_int_2} with 
\begin{equation}
\label{eq_f_new}
f=\sum_{k=1}^m \varepsilon_k f_k, \quad f_k\in C^\infty(\p \Omega), \quad \supp(f_k)\subset \Gamma, \quad k=1,\dots, m.
\end{equation}
Then for all $|\varepsilon|$ sufficiently small, the problem \eqref{eq_int_2} has a unique solution $u(\cdot; \varepsilon)\in C^{2,\alpha}(\overline{\Omega})$ close to $\lambda$ in $C^{2,\alpha}(\overline{\Omega})$-topology, which depends holomorphically on  $\varepsilon\in \text{neigh}(0,\C^m)$, with values in $C^{2,\alpha}(\overline{\Omega})$.

Let $\lambda\in \Sigma$ be arbitrary but fixed. We shall use an induction argument on $m\ge 2$ to prove that the equality 
\[
\Lambda_{\gamma_1}^\Gamma\bigg(\lambda+\sum_{k=1}^m \varepsilon_k f_k\bigg)=\Lambda_{\gamma_2}^\Gamma\bigg(\lambda+\sum_{k=1}^m \varepsilon_k f_k\bigg), 
\]
for all $|\varepsilon|$ sufficiently small and all $f_k\in C^\infty(\p \Omega)$, $\supp(f_k)\subset \Gamma$,  $k=1,\dots, m$, 
gives that $\p_z^{m-1}\gamma_1(x,\lambda, 0)=\p_z^{m-1}\gamma_1(x,\lambda, 0)$. 

First let $m=2$ and we proceed to carry out a second order linearization of the partial Dirichlet--to--Neumann map. Let $u_j=u_j(x;\varepsilon)\in C^{2,\alpha}(\overline{\Omega})$  be the unique solution close to $\lambda$ in $C^{2,\alpha}(\overline{\Omega})$-topology of the Dirichlet problem, 
\begin{equation}
\label{eq_5_1}
\begin{cases}
\Delta u_j+\div \big( \sum_{k=1}^\infty \p^k_z\gamma_j(x,u_j,0)\frac{(\omega\cdot \nabla u_j)^k}{k!}\nabla u_j \big)=0& \text{in}\quad \Omega,\\
u_j=\lambda+\varepsilon_1f_1+\varepsilon_2f_2 & \text{on}\quad \p\Omega,
\end{cases}
\end{equation}
for $j=1,2$. Applying $\p_{\varepsilon_l}|_{\varepsilon=0}$, $l=1,2$, to \eqref{eq_5_1}, and  using that  $u_j(x,0)=\lambda$, we get 
\begin{equation}
\label{eq_5_2}
\begin{cases}
\Delta v_j^{(l)}=0& \text{in}\quad \Omega,\\
v_j^{(l)}=f_l & \text{on}\quad \p\Omega,
\end{cases}
\end{equation}
where $v_j^{(l)}=\p_{\varepsilon_l}u_j|_{\varepsilon=0}$. It follows that $v^{(l)}:=v^{(l)}_1=v^{(l)}_2\in C^\infty(\overline{\Omega})$. 

Applying $\p_{\varepsilon_1}\p_{\varepsilon_2}|_{\varepsilon=0}$ to  \eqref{eq_5_1} and letting $w_j=\p_{\varepsilon_1}\p_{\varepsilon_2}u_j|_{\varepsilon=0}$,  we obtain that 
\begin{equation}
\label{eq_5_3}
\begin{cases}
\Delta w_j +\div\big(\p_z\gamma_j(x,\lambda, 0) ((\omega\cdot \nabla v^{(1)})\nabla v^{(2)}+ (\omega\cdot \nabla v^{(2)})\nabla v^{(1)} ) \big) =0& \text{in}\quad \Omega,\\
w_j=0 & \text{on}\quad \p\Omega,
\end{cases}
\end{equation}
$j=1,2$. 

The fact that $\Lambda_{\gamma_1}^\Gamma(\lambda+\varepsilon_1f_1+\varepsilon_2f_2)=\Lambda_{\gamma_1}^\Gamma(\lambda+\varepsilon_1f_1+\varepsilon_2f_2)$ for  all small $\varepsilon$, and all $f_1,f_2\in C^\infty(\p \Omega)$ with $\supp(f_1),\supp(f_2)\subset\Gamma$, gives that 
\begin{equation}
\label{eq_5_4}
\begin{aligned}
\bigg(1+ \sum_{k=1}^\infty \p^k_z\gamma_1(x,u_1,0)&\frac{(\omega\cdot \nabla u_1)^k}{k!}\bigg)\p_\nu u_1\bigg|_{\Gamma} \\
&=
\bigg(1+ \sum_{k=1}^\infty \p^k_z\gamma_2(x,u_2,0)\frac{(\omega\cdot \nabla u_2)^k}{k!}\bigg)\p_\nu u_2\bigg|_{\Gamma}.
\end{aligned}
\end{equation}
An application of $\p_{\varepsilon_1}\p_{\varepsilon_2}|_{\varepsilon=0}$ to  \eqref{eq_5_4} yields that 
\begin{equation}
\label{eq_5_5}
\begin{aligned}
(\p_\nu w_1-\p_\nu w_2)|_{\Gamma} 
+&(\p_z \gamma_1(x,\lambda,0)- \p_z \gamma_2(x,\lambda,0))\\
 &\times \big((\omega\cdot \nabla v^{(1)})\p_\nu v^{(2)}+ (\omega\cdot \nabla v^{(2)})\p_\nu v^{(1)}  \big)\big|_{\Gamma}=0.
\end{aligned}
\end{equation}
Multiplying the difference of two equations in \eqref{eq_5_3} by $v^{(3)}\in C^\infty(\overline{\Omega})$ harmonic in $\Omega$, integrating over $\Omega$, using Green's formula and \eqref{eq_5_5}, we obtain that 
\begin{equation}
\label{eq_5_6}
\begin{aligned}
&\int_\Omega (\p_z\gamma_1(x,\lambda, 0)- \p_z\gamma_2(x,\lambda, 0)) ((\omega\cdot \nabla v^{(1)})\nabla v^{(2)}+ (\omega\cdot \nabla v^{(2)})\nabla v^{(1)} ) \cdot \nabla v^{(3)}dx\\
&=\int_{\p \Omega\setminus \Gamma}(\p_z\gamma_1(x,\lambda, 0)- \p_z\gamma_2(x,\lambda, 0)) ((\omega\cdot \nabla v^{(1)})\p_\nu v^{(2)}+ (\omega\cdot \nabla v^{(2)})\p_\nu v^{(1)} ) v^{(3)}dS\\
&+\int_{\p \Omega\setminus \Gamma}(\p_\nu w_1-\p_\nu w_2)v^{(3)}dS=0,
\end{aligned}
\end{equation}
provided that $\supp(v^{(3)}|_{\p \Omega})\subset \Gamma$.
It follows from \eqref{eq_5_6} that 
\begin{equation}
\label{eq_5_7}
\int_\Omega (\p_z\gamma_1(x,\lambda, 0)- \p_z\gamma_2(x,\lambda, 0)) ((\omega\cdot \nabla v^{(1)})\nabla v^{(2)}+ (\omega\cdot \nabla v^{(2)})\nabla v^{(1)} ) \cdot \nabla v^{(3)}dx=0,
\end{equation}
for all $v^{(l)}\in C^\infty(\overline{\Omega})$ harmonic in $\Omega$ such that $\supp(v^{(l)}|_{\p \Omega})\subset \Gamma$, $l=1,2,3$. An application of Theorem \ref{thm_main_2} with $m=2$ allows us to conclude that $\p_z\gamma_1(\cdot,\lambda, 0)=\p_z\gamma_2(\cdot,\lambda, 0)$ in $\Omega$.   Now as $\lambda\in \Sigma$ is arbitrary and the functions $\C\ni\tau\to \p_z\gamma_j(x,\tau,0)$, $j=1,2$, are holomorphic, by the uniqueness properties of holomorphic functions, we have $\p_z\gamma_1(\cdot,\cdot, 0)=\p_z\gamma_2(\cdot,\cdot, 0)$ in $\overline{\Omega}\times \C$.

Let $m\ge 3$ and assume that 
\begin{equation}
\label{eq_5_8_0}
\p_z^k\gamma_1(\cdot,\cdot, 0)=\p_z^k\gamma_2(\cdot,\cdot, 0)\text{ in }\overline{\Omega}\times \C,
\end{equation}
 for all $k=1,\dots, m-2$. Let $\lambda\in \Sigma$ be arbitrary but fixed. To prove that $\p_z^{m-1}\gamma_1(\cdot,\lambda, 0)=\p_z^{m-1}\gamma_2(\cdot,\lambda, 0)$ in $\overline{\Omega}$, we carry out the $m$th order linearization of the partial Dirichlet--to--Neumann map. 
In doing so, we let $u_j=u_j(x;\varepsilon)\in C^{2,\alpha}(\overline{\Omega})$ be the unique solution close to $\lambda$ in $C^{2,\alpha}(\overline{\Omega})$-topology of the Dirichlet problem, 
\begin{equation}
\label{eq_5_8}
\begin{cases}
\Delta u_j+\div \big( \sum_{k=1}^\infty \p^k_z\gamma_j(x,u_j,0)\frac{(\omega\cdot \nabla u_j)^k}{k!}\nabla u_j \big)=0& \text{in}\quad \Omega,\\
u_j=\lambda+\varepsilon_1f_1+\dots+ \varepsilon_mf_m & \text{on}\quad \p\Omega,
\end{cases}
\end{equation}
for $j=1,2$.  We shall next apply $\p_{\varepsilon_1}\dots\p_{\varepsilon_m}|_{\varepsilon=0}$ to \eqref{eq_5_8}. To this end, we first note that $\p_{\varepsilon_1}\dots\p_{\varepsilon_m}(\sum_{k=m}^\infty \p^k_z\gamma_j(x,u_j,0)\frac{(\omega\cdot \nabla u_j)^k}{k!}\nabla u_j )$ is a sum of terms each of them containing positive powers of $\nabla u_j$, which vanishes when $\varepsilon=0$. The only term in $\p_{\varepsilon_1}\dots\p_{\varepsilon_m} ( \p^{m-1}_z\gamma_j(x,u_j,0)\frac{(\omega\cdot \nabla u_j)^{m-1}}{(m-1)!}\nabla u_j)$  which does not contain a positive power of $\nabla u_j$ is 
\begin{equation}
\label{eq_5_9}
  \p^{m-1}_z\gamma_j(x,u_j,0) \bigg(\sum_{k=1}^m  \prod_{r=1,r\ne k}^m (\omega\cdot \nabla \p_{\varepsilon_r}u_j)\nabla \p_{\varepsilon_k}u_j\bigg).
\end{equation}
Finally, the expression $\p_{\varepsilon_1}\dots\p_{\varepsilon_m}(\sum_{k=1}^{m-2}\p^k_z\gamma_j(x,u_j,0)\frac{(\omega\cdot \nabla u_j)^k}{k!}\nabla u_j )|_{\varepsilon=0}$ is independent of $j=1,2$. Indeed, this follows from \eqref{eq_5_8_0},  the fact that this expression contains only the derivatives of $u_j$ of the form 
$\p^s_{\varepsilon_{l_1},\dots, \varepsilon_{l_s}}u_j|_{\varepsilon=0}$ with $s=1,\dots, m-1$, $\varepsilon_{l_1},\dots, \varepsilon_{l_s}\in \{\varepsilon_{1},\dots, \varepsilon_{m}\}$, and the fact that  
\begin{equation}
\label{eq_5_10}
\p^s_{\varepsilon_{l_1},\dots, \varepsilon_{l_s}}u_1|_{\varepsilon=0}= \p^s_{\varepsilon_{l_1},\dots, \varepsilon_{l_s}}u_2|_{\varepsilon=0},
\end{equation}
 for $s=1,\dots, m-1$, $\varepsilon_{l_1},\dots, \varepsilon_{l_s}\in \{\varepsilon_{1},\dots, \varepsilon_{m}\}$. The latter can be seen by induction on $s$, applying the operator $\p^s_{\varepsilon_{l_1},\dots, \varepsilon_{l_s}}|_{\varepsilon=0}$ to \eqref{eq_5_8} and using \eqref{eq_5_8_0} as well as the unique solvability of the Dirichlet problem for the Laplacian. 
Thus, an  application $\p_{\varepsilon_1}\dots\p_{\varepsilon_m}|_{\varepsilon=0}$ to \eqref{eq_5_8} gives  
\begin{equation}
\label{eq_5_11}
 \begin{cases}
\Delta w_j+\div \big(  \p^{m-1}_z\gamma_j(x,\lambda,0) \big(\sum_{k=1}^m  \prod_{r=1,r\ne k}^m (\omega\cdot \nabla v^{(r)})\nabla v^{(k)}\big)\big)=H_m& \text{in}\quad \Omega,\\
w_j=0 & \text{on}\quad \p\Omega,
\end{cases}
\end{equation}
cf. \eqref{eq_5_9}. Here $w_j=\p_{\varepsilon_1}\dots\p_{\varepsilon_m}u_j|_{\varepsilon=0}$ and 
\[
H_m(x,\lambda):=-\div\bigg(\p_{\varepsilon_1}\dots\p_{\varepsilon_m}\bigg(\sum_{k=1}^{m-2}\p^k_z\gamma_j(x,u_j,0)\frac{(\omega\cdot \nabla u_j)^k}{k!}\nabla u_j \bigg)\bigg|_{\varepsilon=0}\bigg).
\]

The fact that $\Lambda_{\gamma_1}^\Gamma(\lambda+\varepsilon_1f_1+\dots+\varepsilon_mf_m)=\Lambda_{\gamma_1}^\Gamma(\lambda+\varepsilon_1f_1+\dots+\varepsilon_mf_m)$ for  all small $\varepsilon$ and all $f_k\in C^\infty(\p \Omega)$ with $\supp(f_k),\subset\Gamma$, $k=1,\dots, m$, yields \eqref{eq_5_4}.  Applying of $\p_{\varepsilon_1}\dots\p_{\varepsilon_m}|_{\varepsilon=0}$ to  \eqref{eq_5_4}, using \eqref{eq_5_8_0} and \eqref{eq_5_10}, we obtain that 
\begin{equation}
\label{eq_5_12}
\begin{aligned}
(\p_\nu w_1&-\p_\nu w_2)|_{\Gamma} \\
&+(\p_z^{m-1} \gamma_1(x,\lambda,0)-\p_z^{m-1} \gamma_2(x,\lambda,0)) \bigg(\sum_{k=1}^m  \prod_{r=1,r\ne k}^m (\omega\cdot \nabla v^{(r)})\p_\nu v^{(k)}\bigg)\bigg|_{\Gamma}=0.
\end{aligned}
\end{equation}

Using \eqref{eq_5_11}, \eqref{eq_5_12}, and proceeding as in the case $m=2$, we get 
\begin{equation}
\label{eq_5_13}
\int_\Omega (  \p^{m-1}_z\gamma_1(x,\lambda,0) - \p^{m-1}_z\gamma_1(x,\lambda,0) )\bigg(\sum_{k=1}^m  \prod_{r=1,r\ne k}^m (\omega\cdot \nabla v^{(r)})\nabla v^{(k)}\bigg)\cdot \nabla v^{(m+1)}dx=0,
\end{equation}
for all $v^{(l)}\in C^\infty(\overline{\Omega})$ harmonic in $\Omega$ such that $\supp(v^{(l)}|_{\p \Omega})\subset \Gamma$, $l=1,\dots, m+1$. Applying Theorem \ref{thm_main_2}, we conclude that $\p_z^{m-1}\gamma_1(\cdot,\lambda, 0)=\p_z^{m-1}\gamma_2(\cdot,\lambda, 0)$ in $\overline{\Omega}$. Now as $\lambda\in \Sigma$ is arbitrary and the functions $\C\ni\tau\to \p^{m-1}_z\gamma_j(x,\tau,0)$, $j=1,2$, are holomorphic, we have $\p^{m-1}_z\gamma_1(\cdot,\cdot, 0)=\p^{m-1}_z\gamma_2(\cdot,\cdot, 0)$ in $\overline{\Omega}\times \C$. This completes the proof of Theorem \ref{thm_main}.

\section{Proof of Theorem \ref{thm_main_semilinear}}

\label{sec_Thm_main_semilinear}

First it follows from (a) and (b) that   $\gamma$ can be expanded into the following power series, 
\begin{equation}
\label{eq_int_0_2}
\gamma(x,\lambda)=1+\sum_{k=1}^\infty \p^k_\lambda \gamma(x,0)\frac{\lambda^k}{k!}, \quad \p^k_\lambda\gamma(x,0)\in C^{1,\alpha}(\overline{\Omega}), \quad \lambda\in \C,
\end{equation}
converging in the $C^{1,\alpha}(\overline{\Omega})$ topology. 

Let $\varepsilon=(\varepsilon_1, \dots, \varepsilon_m)\in \C^m$, $m\ge 2$, and consider the Dirichlet problem 
\eqref{eq_int_0_1} with $f$ given by \eqref{eq_f_new}. For all $|\varepsilon|$ sufficiently small, the problem \eqref{eq_int_0_1} has a unique small solution $u(\cdot; \varepsilon)\in C^{2,\alpha}(\overline{\Omega})$, which depends holomorphically on  $\varepsilon\in \text{neigh}(0,\C^m)$.  

As in the proof of Theorem \ref{thm_main}, we use an induction argument on $m\ge 2$ to show that $\Lambda_{\gamma_1}^\Gamma=\Lambda_{\gamma_2}^\Gamma$ implies that $\p_\lambda^{m-1}\gamma_1(x, 0)=\p_\lambda^{m-1}\gamma_1(x, 0)$. 

First let $m=2$ and we perform a second order linearization of the partial Dirichlet--to--Neumann map.  Let $u_j=u_j(x;\varepsilon)\in C^{2,\alpha}(\overline{\Omega})$  be the unique solution small solution of the Dirichlet problem, 
\begin{equation}
\label{eq_8_1}
\begin{cases}
\Delta u_j+\div \big( \sum_{k=1}^\infty \p^k_\lambda\gamma_j(x,0)\frac{ u_j^k}{k!}\nabla u_j \big)=0& \text{in}\quad \Omega,\\
u_j=\varepsilon_1f_1+\varepsilon_2f_2 & \text{on}\quad \p\Omega,
\end{cases}
\end{equation}
for $j=1,2$. Applying $\p_{\varepsilon_l}|_{\varepsilon=0}$, $l=1,2$, to \eqref{eq_8_1}, and  using that  $u_j(x,0)=0$, we see that  
\begin{equation}
\label{eq_8_2}
\begin{cases}
\Delta v_j^{(l)}=0& \text{in}\quad \Omega,\\
v_j^{(l)}=f_l & \text{on}\quad \p\Omega,
\end{cases}
\end{equation}
where $v_j^{(l)}=\p_{\varepsilon_l}u_j|_{\varepsilon=0}$. We have therefore $v^{(l)}:=v^{(l)}_1=v^{(l)}_2\in C^\infty(\overline{\Omega})$. 

Applying $\p_{\varepsilon_1}\p_{\varepsilon_2}|_{\varepsilon=0}$ to  \eqref{eq_8_1} and setting $w_j=\p_{\varepsilon_1}\p_{\varepsilon_2}u_j|_{\varepsilon=0}$,  we get 
\begin{equation}
\label{eq_8_3}
\begin{cases}
\Delta w_j +\div\big(\p_\lambda\gamma_j(x, 0) ( v^{(1)}\nabla v^{(2)}+  v^{(2)}\nabla v^{(1)} ) \big) =0& \text{in}\quad \Omega,\\
w_j=0 & \text{on}\quad \p\Omega,
\end{cases}
\end{equation}
$j=1,2$. 
The fact that $\Lambda_{\gamma_1}^\Gamma(\varepsilon_1f_1+\varepsilon_2f_2)=\Lambda_{\gamma_1}^\Gamma(\varepsilon_1f_1+\varepsilon_2f_2)$ for  all small $\varepsilon$, and all $f_1,f_2\in C^\infty(\p \Omega)$ with $\supp(f_1),\supp(f_2)\subset\Gamma$, implies that 
\begin{equation}
\label{eq_8_4}
\bigg(1+ \sum_{k=1}^\infty \p^k_\lambda\gamma_1(x,0)\frac{ u_1^k}{k!}\bigg)\p_\nu u_1\bigg|_{\Gamma} =
\bigg(1+ \sum_{k=1}^\infty \p^k_z\gamma_2(x,0)\frac{ u_2^k}{k!}\bigg)\p_\nu u_2\bigg|_{\Gamma}.
\end{equation}
Applying  $\p_{\varepsilon_1}\p_{\varepsilon_2}|_{\varepsilon=0}$ to  \eqref{eq_8_4}, we get  
\begin{equation}
\label{eq_8_5}
\begin{aligned}
(\p_\nu w_1-\p_\nu w_2)|_{\Gamma} 
+(\p_\lambda \gamma_1(x,0)- \p_\lambda \gamma_2(x, 0))
  \big( v^{(1)}\p_\nu v^{(2)}+ v^{(2)}\p_\nu v^{(1)}  \big)\big|_{\Gamma}=0.
\end{aligned}
\end{equation}
Multiplying the difference of two equations in \eqref{eq_8_3} by $v^{(3)}\in C^\infty(\overline{\Omega})$ harmonic in $\Omega$, integrating over $\Omega$, using Green's formula and \eqref{eq_8_5}, we obtain that 
\begin{equation}
\label{eq_8_6}
\begin{aligned}
&\int_\Omega (\p_\lambda\gamma_1(x, 0)- \p_\lambda\gamma_2(x, 0)) (v^{(1)}\nabla v^{(2)}+  v^{(2)}\nabla v^{(1)} ) \cdot \nabla v^{(3)}dx\\
&=\int_{\p \Omega\setminus \Gamma}(\p_\lambda\gamma_1(x, 0)- \p_\lambda\gamma_2(x, 0)) (v^{(1)}\p_\nu v^{(2)}+  v^{(2)}\p_\nu v^{(1)} ) v^{(3)}dS\\
&+\int_{\p \Omega\setminus \Gamma}(\p_\nu w_1-\p_\nu w_2)v^{(3)}dS=0,
\end{aligned}
\end{equation}
provided that $\supp(v^{(3)}|_{\p \Omega})\subset \Gamma$. Thus,  \eqref{eq_8_6} gives that 
\[
\int_\Omega (\p_\lambda\gamma_1(x, 0)- \p_\lambda\gamma_2(x, 0)) ( v^{(1)}\nabla v^{(2)}+  v^{(2)}\nabla v^{(1)} ) \cdot \nabla v^{(3)}dx=0,
\]
for all $v^{(l)}\in C^\infty(\overline{\Omega})$ harmonic in $\Omega$ such that $\supp(v^{(l)}|_{\p \Omega})\subset \Gamma$, $l=1,2,3$.  By Theorem \ref{thm_main_2_semilinear} with $m=2$, we get $\p_\lambda\gamma_1(\cdot, 0)=\p_\lambda\gamma_2(\cdot, 0)$ in $\overline{\Omega}$.

Let $m\ge 3$ and assume that $\p_\lambda^k\gamma_1(\cdot, 0)=\p_\lambda^k\gamma_2(\cdot, 0)\text{ in }\overline{\Omega}$, 
for all $k=1,\dots, m-2$.  To prove that $\p_\lambda^{m-1}\gamma_1(\cdot, 0)=\p_\lambda^{m-1}\gamma_2(\cdot,\cdot, 0)$ in $\overline{\Omega}$, we perform the $m$th order linearization of the partial Dirichlet--to--Neumann map.  
In doing so, we let $u_j=u_j(x;\varepsilon)\in C^{2,\alpha}(\overline{\Omega})$ be the unique small solution of the Dirichlet problem, 
\begin{equation}
\label{eq_8_8}
\begin{cases}
\Delta u_j+\div \big( \sum_{k=1}^\infty \p^k_\lambda\gamma_j(x,0)\frac{ u_j^k}{k!}\nabla u_j \big)=0& \text{in}\quad \Omega,\\
u_j=\varepsilon_1f_1+\dots+ \varepsilon_mf_m & \text{on}\quad \p\Omega,
\end{cases}
\end{equation}
for $j=1,2$. Applying $\p_{\varepsilon_1}\dots\p_{\varepsilon_m}|_{\varepsilon=0}$ to \eqref{eq_8_8}, and arguing as in Theorem \ref{thm_main}, we obtain that 
\begin{equation}
\label{eq_8_11}
 \begin{cases}
\Delta w_j+\div \big(  \p^{m-1}_\lambda\gamma_j(x,0) \big(\sum_{k=1}^m  \prod_{r=1,r\ne k}^m v^{(r)}\nabla v^{(k)}\big)\big)=H_m& \text{in}\quad \Omega,\\
w_j=0 & \text{on}\quad \p\Omega.
\end{cases}
\end{equation}
Here $w_j=\p_{\varepsilon_1}\dots\p_{\varepsilon_m}u_j|_{\varepsilon=0}$ and 
\[
H_m(x):=-\div\bigg(\p_{\varepsilon_1}\dots\p_{\varepsilon_m}\bigg(\sum_{k=1}^{m-2}\p^k_\lambda\gamma_j(x, 0)\frac{ u_j^k}{k!}\nabla u_j \bigg)\bigg|_{\varepsilon=0}\bigg),
\] 
which is independent of $j$.

Now the equality $\Lambda_{\gamma_1}^\Gamma(\varepsilon_1f_1+\dots+\varepsilon_mf_m)=\Lambda_{\gamma_1}^\Gamma(\varepsilon_1f_1+\dots+\varepsilon_mf_m)$ for  all small $\varepsilon$ and all $f_k\in C^\infty(\p \Omega)$ with $\supp(f_k),\subset\Gamma$, $k=1,\dots, m$, implies \eqref{eq_8_4}.   Applying of $\p_{\varepsilon_1}\dots\p_{\varepsilon_m}|_{\varepsilon=0}$ to  \eqref{eq_8_4}, we obtain that 
\begin{equation}
\label{eq_8_12}
\begin{aligned}
(\p_\nu w_1&-\p_\nu w_2)|_{\Gamma} +(\p_\lambda^{m-1} \gamma_1(x,0)-\p_\lambda^{m-1} \gamma_2(x,0)) \bigg(\sum_{k=1}^m  \prod_{r=1,r\ne k}^m  v^{(r)}\p_\nu v^{(k)}\bigg)\bigg|_{\Gamma}=0.
\end{aligned}
\end{equation}

Proceeding as in the case $m=2$, and using \eqref{eq_8_11}, \eqref{eq_8_12}, we get 
\[
\int_\Omega (  \p^{m-1}_\lambda\gamma_1(x,0) - \p^{m-1}_\lambda\gamma_1(x,,0) )\bigg(\sum_{k=1}^m  \prod_{r=1,r\ne k}^m  v^{(r)}\nabla v^{(k)}\bigg)\cdot \nabla v^{(m+1)}dx=0,
\]
for all $v^{(l)}\in C^\infty(\overline{\Omega})$ harmonic in $\Omega$ such that $\supp(v^{(l)}|_{\p \Omega})\subset \Gamma$, $l=1,\dots, m+1$.  An application of  Theorem \ref{thm_main_2_semilinear} allows us to conclude that $\p_\lambda^{m-1}\gamma_1(\cdot, 0)=\p_\lambda^{m-1}\gamma_2(\cdot,0)$ in $\overline{\Omega}$. This completes the proof of Theorem \ref{thm_main_semilinear}.

\begin{appendix}
\section{Proof of Theorem \ref{thm_main} in the case of full data}
\label{sec_Thm_main_full_data}

Note that the result of Theorem \ref{thm_main} is new even in the case of full data, i.e. $\Gamma=\p \Omega$,  and the purpose of this appendix is to present an alternative simple proof in this case.

Using the linearization of the Dirichlet--to--Neumann map $\Lambda_{\gamma}^{\p \Omega}$, we shall see below that the proof of Theorem \ref{thm_main} in the full data case   will be a consequence of the following density result.
\begin{prop}
\label{prop_density}
Let $\Omega\subset \R^n$, $n\ge 2$, be a bounded open set with $C^\infty$ boundary, let $\omega\in \mathbb{S}^{n-1}$ be fixed and let $m=2,3, \dots,$ be fixed. Let $f\in L^\infty(\Omega)$ be such that 
\begin{equation}
\label{eq_app_1}
\int_\Omega f  (\omega\cdot \nabla v_1)^{m-1} \nabla v_1 \cdot \nabla v_2 dx=0,
\end{equation}
for all  functions $v_1, v_2\in C^\infty(\overline{\Omega})$ harmonic in $\Omega$.  Then $f=0$ in $\Omega$.  
\end{prop}

\begin{proof}
Let $\xi\in \mathbb{S}^{n-1}$ and consider $k\in \mathbb{S}^{n-1}$ such that $\xi\cdot k=0$. Let $h>0$.  Setting 
\[
v_1(x)=e^{\frac{1}{h}x\cdot(k+i\xi)}, \quad v_2(x)=e^{\frac{m}{h}x\cdot(-k+i\xi)}, 
\] 
so that  $v_1,v_2\in C^\infty(\R^n)$ and harmonic.  Substituting $v_1$ and $v_2$ into \eqref{eq_app_1} and using that $(k+i\xi)\cdot(-k+i\xi)=-2$, we get 
\[
(\omega\cdot (k+i\xi))^{m-1} \int_\Omega f(x)   e^{\frac{2m }{h} i x\cdot\xi} dx=0,
\]
and therefore, we have 
\[
 \int_\Omega f(x)   e^{\frac{2m }{h} ix\cdot\xi} dx=0,
\]
for all $\xi\in \mathbb{S}^{n-1}$, $\xi\cdot \omega\ne 0$, and all $h>0$.  Hence, $f=0$. 
\end{proof}

Let $\lambda\in \Sigma$ be arbitrary but fixed. We shall use an induction argument on $m\ge 2$ to prove that the equality 
\[
\Lambda_{\gamma_1}^\Gamma\bigg(\lambda+\sum_{k=1}^m \varepsilon_k f_k\bigg)=\Lambda_{\gamma_2}^\Gamma\bigg(\lambda+\sum_{k=1}^m \varepsilon_k f_k\bigg), 
\]
for all $|\varepsilon|$ sufficiently small and all $f_k\in C^\infty(\p \Omega)$, $\supp(f_k)\subset \Gamma$,  $k=1,\dots, m$, 
gives that $\p_z^{m-1}\gamma_1(x,\lambda, 0)=\p_z^{m-1}\gamma_1(x,\lambda, 0)$. 

First when $m=2$, taking $v^{(1)}=v^{(2)}$ in \eqref{eq_5_7} and using Proposition \ref{prop_density} with $m=2$, we get $\p_z\gamma_1(\cdot,\lambda, 0)=\p_z\gamma_2(\cdot,\lambda, 0)$ in $\Omega$. Now as $\lambda\in \Sigma$ is arbitrary,  we have $\p_z\gamma_1(\cdot,\cdot, 0)=\p_z\gamma_2(\cdot,\cdot, 0)$ in $\overline{\Omega}\times \C$.

Let $m=3,4, \dots$. Let $\lambda\in \Sigma$ be arbitrary but fixed.  Letting $v^{(1)}=\dots=v^{(m)}$ in \eqref{eq_5_13} and using Proposition \ref{prop_density}, we see that $\p_z^{m-1}\gamma_1(\cdot,\lambda, 0)=\p_z^{m-1}\gamma_2(\cdot,\lambda, 0)$ in $\Omega$. Again, as $\lambda\in \Sigma$ is arbitrary, we get $\p^{m-1}_z\gamma_1(\cdot,\cdot, 0)=\p^{m-1}_z\gamma_2(\cdot,\cdot, 0)$ in $\overline{\Omega}\times \C$.  This completes the proof of Theorem \ref{thm_main}  in the full data case.

\section{Well-posedness of the Dirichlet problem for a quasilinear conductivity equation}
\label{App_well_posedness}

In this appendix we shall recall a standard argument for showing the well-posedness of the Dirichlet problem for a quasilinear conductivity equation. 

Let $\Omega\subset \R^n$, $n\ge 2$, be a bounded open set with $C^\infty$ boundary. Let $k\in \N\cup \{0\}$ and $0<\alpha<1$ and let $C^{k,\alpha}(\overline{\Omega})$ be the standard H\"older space  on $\Omega$, see \cite{Krup_Uhlmann_2}, \cite{Hormander_1976}.  We observe that $C^{k,\alpha}(\overline{\Omega})$ is an algebra under pointwise multiplication, with 
\begin{equation}
\label{eq_app_ref_01}
\|uv\|_{C^{k, \alpha}(\overline{\Omega})} \le C\big(\|u\|_{C^{k,\alpha}(\overline{\Omega})} \|v\|_{L^\infty(\Omega)}+\|u\|_{L^\infty(\Omega)} \|v\|_{C^{k,\alpha}(\overline{\Omega})} \big), \quad u,v\in C^{k,\alpha}(\overline{\Omega}),
\end{equation}
see \cite[Theorem A.7]{Hormander_1976}. We write $C^{\alpha}(\overline{\Omega})=C^{0,\alpha}(\overline{\Omega})$.

Let $\omega\in\mathbb{S}^{n-1}=\{\omega\in  \R^n, |\omega|=1\}$, be fixed. Consider the Dirichlet problem for the following isotropic quasilinear conductivity equation, 
\begin{equation}
\label{eq_app_ref_1}
\begin{cases} 
\div (\gamma(x,u,\omega\cdot \nabla u)\nabla u)=0& \text{in}\quad \Omega,\\
u=\lambda+f & \text{on}\quad \p \Omega,
\end{cases}
\end{equation}
with $\lambda\in \C$.  We assume that the function $\gamma: \overline{\Omega}\times \C\times \C\to \C$ satisfies  the following conditions,   
\begin{itemize}
\item[(i)] the map $\C\times \C\ni (\tau, z)\mapsto \gamma(\cdot, \tau,z)$ is holomorphic with values in $C^{1,\alpha}(\overline{\Omega})$ with some $0<\alpha<1$, 
\item[(ii)] $\gamma(x,0,0)=1$.
\end{itemize}
It follows from (i) and (ii) that    $\gamma$ can be expand into a power series 
\begin{equation}
\label{eq_app_gamma}
\gamma(x,\tau, z)= 1+\sum_{j+k\ge 1, j\ge 0, k\ge 0}
\p_\tau^j\p^k_z\gamma(x,0,0)\frac{\tau^j z^k}{j!k!},\quad 
 \p_\tau^j\p^k_z\gamma(x,0,0)\in C^{1,\alpha}(\overline{\Omega}),
\end{equation}
converging in the $C^{1,\alpha}(\overline{\Omega})$ topology.

We have the following result.

\begin{thm}
\label{thm_well-posedness}
Let $\lambda\in \C$ be fixed. Then  under the above assumptions, there exist $\delta>0$, $C>0$ such that for any $f\in  B_{\delta}(\p \Omega):=\{f\in C^{2,\alpha}(\p \Omega): \|f\|_{C^{2,\alpha}(\p \Omega)}< \delta\}$, the problem \eqref{eq_app_ref_1} has a solution $u=u_{\lambda, f}\in C^{2,\alpha}(\overline{\Omega})$ which satisfies
\[
\|u-\lambda \|_{C^{2,\alpha}(\overline{\Omega})}\le C\|f\|_{C^{2,\alpha}(\p \Omega)}.
\]
The solution $u$ is unique within the class $\{u\in C^{2,\alpha}(\overline{\Omega}): \|u-\lambda\|_{C^{2,\alpha}(\overline{\Omega})}< C\delta \}$ and
 it is depends holomorphically on $f\in B_\delta(\p \Omega)$. Furthermore, the map
\[
B_\delta(\p \Omega)\to C^{1,\alpha}(\overline{\Omega}), \quad f\mapsto \p_\nu u|_{\p \Omega}
\]
is holomorphic.
\end{thm}

\begin{proof}
Let $\lambda\in \C$ be fixed, and let 
\[
B_1=C^{2,\alpha}(\p \Omega), \quad  B_2=C^{2,\alpha}(\overline{\Omega}), \quad B_3=C^{\alpha}(\overline{\Omega})\times C^{2,\alpha}(\p \Omega). 
\]
Consider the map,
\begin{equation}
\label{eq_app_ref_2}
F:B_1\times B_2\to B_3, \quad F(f,u)=(\div (\gamma(x,u,\omega\cdot \nabla u)\nabla u), u|_{\p \Omega}-\lambda-f).
\end{equation}
Following  \cite{LLLS}, we shall make use of the implicit function theorem for holomorphic maps between complex Banach spaces, see \cite[p. 144]{Poschel_Trub_book}.  First we check that $F$ enjoys the mapping property \eqref{eq_app_ref_2}. To that end in view of the fact that  $C^{1,\alpha}(\overline{\Omega})$ is an  algebra under pointwise multiplication, we only need to show that $\gamma(x,u,\omega\cdot \nabla u)\in C^{1,\alpha}(\overline{\Omega})$. In doing so, by Cauchy's estimates, we get 
\begin{equation}
\label{eq_app_ref_3}
\|\p_\tau^j\p_z^k \gamma(x,0,0)\|_{C^{1,\alpha}(\overline{\Omega})}\le \frac{j! k!}{R_1^jR_2^k}\sup_{|\tau|=R_1, |z|=R_2}\|\gamma(\cdot,\tau,  z)\|_{C^{1,\alpha}(\overline{\Omega})}, \quad R_1, R_2>0,
\end{equation}
for all $j\ge 0$, $k\ge 0$,  and $j+k\ge 1$. With the help of \eqref{eq_app_ref_01} and \eqref{eq_app_ref_3}, we obtain that 
\begin{equation}
\label{eq_app_ref_4}
\begin{aligned}
\bigg\| \p_\tau^j\p_z^k \gamma(x,0,0)&\frac{u^j(\omega\cdot \nabla u)^k}{j!k!} \bigg\|_{C^{1, \alpha}(\overline{\Omega})}\\
&\le \frac{C^{j+k}}{R_1^jR_2^k}\|u\|^j_{C^{1,\alpha}(\overline{\Omega})} \|\omega\cdot \nabla u\|^k_{C^{1,\alpha}(\overline{\Omega})} \sup_{|\tau|=R_1, |z|=R_2}\|\gamma(\cdot,\tau,  z)\|_{C^{1,\alpha}(\overline{\Omega})}.
\end{aligned}
\end{equation}
Taking $R_1=2C\|u\|_{C^{1,\alpha}(\overline{\Omega})}$ and $R_2=2C\|\omega \cdot \nabla u\|_{C^{1,\alpha}(\overline{\Omega})}$, we see that the series 
\[
\sum_{j+k\ge 1,j\ge 0, k\ge 0}\p_\tau^j\p^k_z\gamma(x,0,0)\frac{u^j (\omega\cdot\nabla u)^k}{j!k!}
\]
 converges in $C^{1, \alpha}(\overline{\Omega})$. Hence, in view of \eqref{eq_app_gamma}, $\gamma(x,u,\omega\cdot \nabla u)\in C^{1,\alpha}(\overline{\Omega})$.

Let us show that  $F$ in \eqref{eq_app_ref_2} is holomorphic.  First $F$ is  locally bounded as it is continuous in $(f,u)$. Hence, we only need to check that $F$ is  weak holomorphic,  see \cite[p. 133]{Poschel_Trub_book}. 
To that end, letting $(f_0, u_0), (f_1,u_1)\in B_1\times B_2$, we show that  the map
\[
 \mu \mapsto F((f_0, u_0)+\mu (f_1,u_1))
\]
is holomorphic in $\C$ with values in $B_3$.  Clearly, we only have to check that the map 
$\mu \mapsto \gamma(x,u_0(x)+\mu u_1(x), \omega\cdot (\nabla u_0(x)+\mu \nabla u_1(x)))$ is holomorphic in $\C$ with values in $C^{1,\alpha} (\overline{\Omega})$. This is a consequence of the fact that the series
\[
\sum_{j+k\ge 1,\j\ge 0, k\ge 0}
\p_\tau^j\p^k_z\gamma(x,0,0)\frac{(u_0+\mu u_1)^j (\omega\cdot\nabla (u_0+\mu u_1))^k}{j!k!} 
\]
converges in $C^{1,\alpha}(\overline{\Omega})$, locally  uniformly in  $\lambda\in \C$, in view of  \eqref{eq_app_ref_4}.

We have $F(0,\lambda)=0$ and the partial differential $\p_u F(0,\lambda): B_2\to B_3$ is given by
\[
\p_u F(0,\lambda)v=(\Delta v , v|_{\p \Omega}).
\]
It follows from \cite[Theorem 6.15]{Gil_Tru_book} that the map  $\p_u F(0,\lambda ): B_2\to B_3$ is a linear isomorphism.

An application of  the implicit function theorem, see \cite[p. 144]{Poschel_Trub_book}, shows that there exists  $\delta>0$ and a unique holomorphic map $S: B_\delta(\p \Omega)\to C^{2,\alpha}(\overline{\Omega})$ such that $S(0)=\lambda$ and $F(f, S(f))=0$ for all $f\in B_\delta(\p \Omega)$. Letting $u=S(f)$ and using that $S$ is Lipschitz continuous and $S(0)=\lambda$, we have
\[
\|u-\lambda\|_{C^{2,\alpha}(\overline{\Omega})}\le C\|f\|_{C^{2,\alpha}(\p \Omega)}.
\]
\end{proof}
\end{appendix}

\section*{Acknowledgements}

The work of Y.K is partially supported by  the French National Research Agency ANR (project MultiOnde) grant ANR-17-CE40-0029.
The research of K.K. is partially supported by the National Science Foundation (DMS 1815922). The research of G.U. is partially supported by NSF, a Walker Professorship at UW and a Si-Yuan Professorship at IAS, HKUST. Part of the work was supported by the NSF grant DMS-1440140 while K.K. and G.U. were in residence at MSRI in Berkeley, California, during Fall 2019 semester.

\end{document}